\documentclass[a4paper,leqno]{amsart}
\usepackage{graphicx, amsmath, amssymb, amsthm, hyperref, verbatim, multicol, caption, subcaption,enumerate, mathtools}

\usepackage{tikz,pgf,tikz-3dplot,tikzscale,pgfplots}
\usetikzlibrary{calc,math,3d}
\usepackage[normalem]{ulem}

\usepackage[english]{babel}
\usepackage[T1]{fontenc}
\usepackage[utf8]{inputenx}

\theoremstyle{plain}
\newtheorem{thm}{Theorem}
\newtheorem{lemm}[thm]{Lemma}

\newtheorem{prop}[thm]{Proposition}
\newtheorem{cor}[thm]{Corollary}

\theoremstyle{definition}

\theoremstyle{remark}
\newtheorem{rem}[thm]{Remark}
\newtheorem{question}{Question}

\DeclareMathOperator{\Id}{id}
\DeclareMathOperator{\loc}{loc}

\DeclareMathOperator{\diam}{diam}

\DeclareMathOperator{\e}{\bf e_3}

\newcommand{\x}{\scalebox{1.2}{$\chi$} } 

\newcommand{\br}{\overline}
\newcommand{\R}{\mathbb R}

\newcommand{\D}{\mathbb D}

\newcommand{\N}{\mathbb N}

\newcommand{\id}{\mathrm{id}}
\newcommand{\inter}{\mathrm{int}}

\numberwithin{equation}{section}
\numberwithin{thm}{section}
\begin{document}
	
	\title[Inverse absolute continuity of quasiconformal mappings]{On the inverse absolute continuity of quasiconformal mappings on hypersurfaces}
	
	\author{Dimitrios Ntalampekos}
	\address{Institute for Mathematical Sciences, Stony Brook University, Stony Brook, NY 11794, USA}
	\email{dimitrios.ntalampekos@stonybrook.edu}
	
	\author{Matthew Romney}
	\address{Department of Mathematics and Statistics, University of Jyv\"askyl\"a, P.O. Box 35 (MaD), FI-40014, University of Jyv\"askyl\"a, Finland}
	\email{matthew.d.romney@jyu.fi}
	
	\thanks{The first author was partially supported by NSF grant DMS-1506099. The second author was supported by the Academy of Finland grant 288501 and by the ERC Starting Grant 713998 GeoMeG} 
	
	\date{\today}

	\begin{abstract}
		We construct quasiconformal mappings $f\colon \mathbb{R}^{3} \rightarrow \mathbb{R}^{3}$ for which there is a Borel set $E \subset \mathbb{R}^2 \times \{0\}$ of positive Lebesgue $2$-measure whose image $f(E)$ has Hausdorff $2$-measure zero. This gives a solution to the open problem of inverse absolute continuity of quasiconformal mappings on hypersurfaces, attributed to Gehring. By implication, our result also answers questions of V\"ais\"al\"a and Astala--Bonk--Heinonen. 
	\end{abstract}
	
	\maketitle
	
	\section{Introduction} \label{sec:introduction}
	
	In this paper, we give a solution to the problem of {\it inverse absolute continuity of quasiconformal mappings on hypersurfaces} for two-dimensional hypersurfaces in $\mathbb{R}^3$. The general problem asks whether a quasiconformal mapping $f\colon \R^{n+1}\to \R^{n+1}$ ($n\geq 2$), restricted to a smooth ($n$-dimensional) hypersurface $Z$ in $\mathbb{R}^{n+1}$, always maps sets of positive Hausdorff $n$-measure to sets of positive Hausdorff $n$-measure. We show, by construction, that this  may fail to be the case when $n=2$. We work with the simplest case of $Z = \mathbb{R}^2 \times \{0\}$. Our construction scheme becomes technically more difficult in higher dimensions, and so in this paper we only consider the case $n=2$. We believe that the main idea should be applicable in all dimensions.
	
	The inverse absolute continuity problem for quasiconformal mappings originated in the 1970s from work of Gehring on distortion of sets of a given Hausdorff dimension under a quasiconformal mapping \cite{Gehr:75}, \cite{Gehr:76}. A statement can be found in Section 10 of the 1982 survey article of Baernstein and Manfredi \cite{BaeMan:82} and in Problem 7.55 of Hayman and Lingham's problem list \cite{HL:19}, although the question is attributed to Gehring. See also \cite[Sec. 1]{ABH:02}, \cite[Sec. 6]{Hei:96}, \cite[Sec. 8]{HeiKos:94}, \cite[Sec. 17]{Raj:17} for discussion around this and related problems. 
	
	Let us recall the fundamental definition. Let $n \geq 2$ be an integer and $U,V \subset \mathbb{R}^n$ be domains. A homeomorphism $f\colon U \rightarrow V$ is  {\it quasiconformal} if it is in the Sobolev space $W_{\loc}^{1,n}(U, \mathbb{R}^n)$ and satisfies the differential inequality
	\begin{equation} \label{equ:analytic_definition}
	\|Df(x)\|^n \leq K J_f(x) 
	\end{equation}
	for a.e.\ $x \in U$ and some $K\geq 1$ independent of $x$. In this case, we also say that $f$ is {\it $K$-quasiconformal}. Here, $\|Df(x)\|$ is the operator norm of the matrix of partial derivatives of $f$, and $J_f$ is the Jacobian of $f$. The smallest constant $K$ for which \eqref{equ:analytic_definition} is satisfied is called the {\it dilatation} of $f$. Geometrically, a quasiconformal mapping takes an infinitesimal ball to a topological ball of uniformly bounded eccentricity.
	
	Equivalently, the homeomorphism $f\colon U \rightarrow V$ is quasiconformal if and only if it is orientation-preserving and there exists a constant $H \geq 1$ such that 
	\begin{equation} \label{equ:qc_definition}
	\limsup_{r \rightarrow 0} \frac{\sup\{|f(x) - f(y)|: |x-y| \leq r \}}{\inf\{|f(x) - f(y)|: |x-y| \geq r\}} \leq H
	\end{equation}
	for all $x \in U$. However, it is the analytic description of quasiconformal mappings in \eqref{equ:analytic_definition} that is most convenient for our purposes and that we will use as our working definition throughout this paper. The standard reference for the theory of Euclidean quasiconformal mappings is the lecture notes of V\"ais\"al\"a \cite{Vais:71}. 
	
	A mapping $f\colon (X,\mu) \rightarrow (Y, \nu)$ between measure spaces is {\it absolutely continuous in measure}, or simply {\it absolutely continuous}, if $\nu(f(E)) = 0$ for all measurable sets $E \subset X$ satisfying $\mu(E) = 0$. The property of absolute continuity in measure is also called {\it Lusin's condition (N)} in the literature. An important fact is that any quasiconformal mapping between domains in $\mathbb{R}^n$, where $n \geq 2$, is absolutely continuous with respect to Lebesgue $n$-measure. See \cite[Theorem 33.2]{Vais:71}.
	
	A version of the inverse absolute continuity problem can be stated for the related notion of quasisymmetric mappings. Quasisymmetric mappings are a generalization of quasiconformal mappings that is well-suited for the metric space setting. Let $(X, d_X)$, $(Y, d_Y)$ be metric spaces. A homeomorphism $f\colon  X \rightarrow Y$ is {\it quasisymmetric} if there is a homeomorphism $\eta\colon [0, \infty) \rightarrow [0, \infty)$ such that 
	\begin{equation} \label{equ:qs_definition}
	\frac{d_Y(f(x), f(y))}{d_Y(f(x), f(z))} \leq \eta \left( \frac{d_X(x,y)}{d_X(x,z)} \right)
	\end{equation}
	for all triples of distinct points $x,y,z \in X$. It is simple to see that the condition \eqref{equ:qs_definition} implies the condition \eqref{equ:qc_definition} in the Euclidean case, so that orientation-preserving quasisymmetric mappings are quasiconformal. In fact, for homeomorphisms of $\mathbb{R}^n$, the converse also holds: quasiconformal mappings are quasisymmetric. The quasisymmetry condition was first introduced by Ahlfors and Beurling in \cite{AB:56}, where they show that any quasiconformal homeomorphism of the upper half-plane induces a quasisymmetric mapping of the real line. They moreover show that a quasisymmetry of the real line may fail to be absolutely continuous with respect to 1-dimensional Lebesgue measure; see also more precise results by Tukia \cite{Tuk:89}. Quasisymmetric mappings were first explicitly studied by Tukia and V\"ais\"al\"a in \cite{TukVai:80}. See also the book of Heinonen \cite[Ch.\ 10-11]{Hei:01} for the basic theory.
	
	The quasisymmetric version of the problem asks, for $n\geq 2$, whether the inverse of a quasisymmetric mapping $f\colon \mathbb{R}^n \rightarrow (X,d)$, where $(X,d)$ is some metric space, must be absolutely continuous with respect to the Hausdorff $n$-measure. These can be found as Questions 15 and 16 in \cite{HeiS:97}, where Question 15 has the additional assumption that $X$ has locally finite Hausdorff $n$-measure. A solution to this problem was recently given by the second named author in \cite{Rom:18}. More precisely, a metric $d$ on $[0,1]^n$ ($n \geq 2$) is constructed for which the identity map $\Id \colon ([0,1]^n,|\cdot|) \rightarrow ([0,1]^n, d)$ is not absolutely continuous. The construction can be carried out so that a set of full measure is mapped to a set of arbitrarily small positive Hausdorff dimension. The basic idea of the present paper is to carry out a similar construction directly in Euclidean space. 
	
	As in \cite{Rom:18}, we give two versions of the main theorem. In the first version, a stronger conclusion is possible if one does not require the image of $\mathbb{R}^2 \times \{0\}$ to have locally finite Hausdorff $2$-measure. The notation $\mathcal{H}^2$ is used here and throughout this paper for Hausdorff $2$-measure in $\R^3$ and for Lebesgue $2$-measure in the plane, which coincides with Hausdorff $2$-measure up to a normalizing factor. 
	
	\begin{thm} \label{thm:main1}
		Let $A \subset \mathbb{R}^2 \times \{0\}$ be a Borel set. There exists a subset $A' \subset \R^2 \times \{0\}$ such that $\mathcal{H}^2(A \setminus A') = 0$ and a quasiconformal mapping $f\colon \mathbb{R}^3 \rightarrow \mathbb{R}^3$ such that $f(A')$ has Hausdorff $2$-measure zero.   
	\end{thm}
	
	\begin{thm} \label{thm:main2}
		Let $A \subset \mathbb{R}^2 \times \{0\}$ be a Borel set, and let $\kappa>0$. There exists a subset $A'' \subset \R^2 \times \{0\}$ such that $\mathcal{H}^2(A \setminus A'') < \kappa$ and a quasiconformal mapping $f\colon \mathbb{R}^3 \rightarrow \mathbb{R}^3$ such that $f(A'')$ has Hausdorff $2$-measure zero and $f(\mathbb{R}^2\times \{0\})$ has locally finite Hausdorff $2$-measure. In fact, $f$ is locally Lipschitz.  
	\end{thm}
	
	It is known from work of Tukia and V\"ais\"al\"a \cite{TukVai:82} that any quasiconformal mapping $f\colon \mathbb{R}^n \rightarrow \mathbb{R}^n$ extends to a quasiconformal mapping $
	\widehat{f}\colon \mathbb{R}^N \rightarrow \mathbb{R}^N$ for all dimensions $N>n$. In particular, for all $n \geq 2$, there exists a quasiconformal mapping $f\colon \mathbb{R}^n \rightarrow \mathbb{R}^n$ such that the conclusions of Theorem \ref{thm:main1} and \ref{thm:main2} hold for $\mathbb{R}^2 \times \{0\}^{n-2}$. 
	
	For a given $\lambda >0$, we may construct $f$ so that it is the identity map on $\{z \in \mathbb{R}^3: d(z,A) \geq \lambda\}$. As for other desirable properties of our construction, one might attempt to carry out this construction so that $f(A')$ has arbitrarily small positive Hausdorff dimension, or so that the dilatation $K$ is arbitrarily close to 1. In fact, a lower bound on the Hausdorff dimension of $f(A')$ is proved by Heinonen and Koskela in \cite[Theorem A]{HeiKos:94}, showing that the first of these goals is impossible. We do not know whether one can construct $f$ so that $f(A')$ has Hausdorff dimension less than 1. Note that, by \cite{Mey:09}, a quasiconformal mapping $f\colon \mathbb{R}^3 \to \mathbb{R}^3$ may map a full-measure subset of $\mathbb{R}^2 \times \{0\}$ onto a set of Hausdorff dimension smaller than that of $f(\mathbb{R}^2 \times \{0\})$. It also appears doubtful that an adaptation of our construction can have dilatation arbitrarily close to 1. This raises the following question. 
	
	\begin{question}
		Does there exist a value $K>1$, depending on $n$, such that the property of inverse absolute continuity on hypersurfaces holds for any $K$-quasiconformal mapping $f\colon \mathbb{R}^{n+1} \to \mathbb{R}^{n+1}$?
	\end{question} 
	
	We conclude the introduction with a discussion of the related literature. Theorem \ref{thm:main2} also resolves three other open problems in the literature, again for the case $n=2$. The first is Question 5.10 of V\"ais\"al\"a in \cite{Vai:81}. This asks whether a quasisymmetric embedding $f\colon \mathbb{R}^n \rightarrow \mathbb{R}^N$ ($N > n \geq 2$) can map a set of positive $n$-measure onto a set of Hausdorff $n$-measure zero. Since the restriction of the map $f$ in Theorem \ref{thm:main2} to the set $\mathbb{R}^2$ is quasisymmetric, we obtain the following corollary.
	
	\begin{cor}
		There exists a quasisymmetric embedding $f\colon \mathbb{R}^2 \rightarrow \mathbb{R}^3$ and a set $A \subset \mathbb{R}^2$ of positive Lebesgue $2$-measure such $f(\mathbb{R}^2)$ has locally finite Hausdorff $2$-measure and $f(A)$ has Hausdorff $2$-measure zero. 
	\end{cor}  
	
	The other two problems are due to Astala, Bonk and Heinonen \cite{ABH:02}. The first (Problem 1.3) concerns the boundary behavior of quasiconformal mappings $f\colon \mathbb{R}_+^{n+1} \rightarrow \mathbb{R}^{n+1}$ belonging to what they term the {\it Riesz class}.  The precise definition of the Riesz class is lengthy, so we omit it here. Problem 1.3 asks whether, for such a mapping $f$, the differential matrix of the induced boundary map on $\mathbb{R}^n \times \{0\}$ must vanish on a set of Lebesgue $n$-measure zero. As explained in \cite{ABH:02}, the mapping in Theorem \ref{thm:main2}, restricted to $\mathbb{R}_+^{3}$, lies in the Riesz class and hence provides a negative answer to this question. The second (Problem 4.2) asks the same question for certain conformal densities on $\mathbb{R}_+^{n+1}$, and again Theorem \ref{thm:main2} gives a negative answer. We refer the interested reader to the paper \cite{ABH:02} for a precise statement of these problems and the relevant definitions. 
	
	A direct source of inspiration for many of the ideas in our construction is a paper of Bishop \cite{Bish:99}, in which a quasiconformal mapping $f\colon \mathbb{R}^3 \rightarrow \mathbb{R}^3$ is constructed with the property that $f(\mathbb{R}^2\times \{0\})$ does not contain any rectifiable curves; see David--Toro \cite{DT:99} for a similar construction. See also the paper of Heinonen \cite{Hei:96} on the related question of boundary absolute continuity of quasiconformal mappings.
	
	A question of a similar nature to the inverse absolute continuity problem was asked by Bishop in \cite[Question 3]{Bish:94}:
	\begin{question}
		If $E\subset \R^2$ is a compact set with positive measure and empty interior, does there exist a homeomorphism $f\colon \R^2 \to \R^2$ that is conformal on $\R^2\setminus E$ and maps $E$ to a set of measure zero?
	\end{question}
	Despite some partial results by Kaufman and Wu \cite{KauWu:96} and the first named author \cite{Nta:18}, the question remains open. Motivated by that question, we ask whether $A'$ can be equal to $A$ in Theorem \ref{thm:main1}:
	\begin{question}
		If $A\subset \R^2 \times \{0\}$ is a compact set with positive $2$-measure and empty interior, does there exist a quasiconformal mapping $f\colon \R^3\to \R^3$ such that $f(A)$ has Hausdorff $2$-measure zero?
	\end{question}

	\subsection{Overview of the construction} 
	
	The heart of our approach is carried out in Section \ref{sec:qc_cylinder}. Here, we construct a family of quasiconformal self-mappings $\{F_k\}_{k \in \mathbb{N}}$ of the cylinder $T= \br {\mathbb{D}} \times [-2,2]$ with various properties, listed out in Proposition \ref{prop:basic_construction}. Here, $\mathbb{D}$ is the unit disk $\{(r,\theta): r <1\}$ (expressed in polar coordinates). The basic idea is that the mappings $F_k$ should fix the boundary $\partial T$ and contract distances on a large subset of $\mathbb{D} \times \{0\}$, while providing enough control on the behavior elsewhere. 
	More precisely, each $F_k$ is a similarity mapping from the disk $\{(r,\theta): r <a\}\times \{0\}$ onto the disk $\{(r,\theta): r <ab\}\times \{0\}$ for appropriate constants $a,b \in (0,1)$.    
	
	In Section \ref{sec:partition_composition}, we describe a scheme for composing rescaled and translated versions of the mappings $F_k$ constructed in Section \ref{sec:qc_cylinder} to give, in the limit, the mapping $f$ described in Theorems \ref{thm:main1} and \ref{thm:main2}. The main difficulty is to compose these mappings in such a way that we maintain a uniform bound on the pointwise dilatation. To achieve this, it it is crucial that the mappings $F_k\colon T\to T$ are very close to being conformal {(i.e., ($1+\eta(k)$)-quasiconformal for a small $\eta(k)>0$, with $\eta(k) \rightarrow 0$ as $k \rightarrow \infty$)} when restricted to an open, full-measure subset of $\br\D\times \{0\}$.
	
	\begin{figure}[t]
		\centering
		\begin{tikzpicture}[scale=.7]

%
%
%
%

\begin{scope}[shift=({-5,7})]
\draw[fill=black!40] (0,0,0) circle (3cm);
\draw[fill=black!20] (0,0,0) circle (2cm);
\end{scope}

\begin{scope}[shift=({5,7})]
\draw[fill=black!40] (0,0,0) circle (3cm);
\draw[fill=black!20] (0,0,0) circle (1.3cm);
\end{scope}

\begin{scope}[shift=({-9,11})]
\draw[fill=black!40] (0,0) rectangle (10,1);
\end{scope}

\draw[->] (-5,7) to [out=30, in =150] (5,7) node[anchor=north] {$B(0,ab)$};
\node (br) at (0,8) {$r\mapsto br$};
\node (Ba) at (-5,6) {$B(0,a)$};
\node[anchor=west] (log) at (-4, 10.3) {$\log$ (unwrap)};
\node[anchor=west] (fold) at (-2.5, 13) {fold and scale};
\node[anchor=west] (wrap) at (5, 10.8) {wrap around};
\node[anchor=west] (D1) at (-2.7,5) {$\mathbb D$};
\node[anchor=west] (D2) at (7.3,5) {$\mathbb D$};

\draw[->] (-5,9.5) to [out=90, in=200] (-4,10.8);
\draw[->] (-3,12.2) to [out=90, in =180] (0,14);
\draw[->] (5,11.5) to (5,9.5);
%
%
%
%

\begin{scope}[shift=({0,14}),scale=.5,line cap=round,line join=round,>=triangle 45,x=1.0cm,y=1.0cm, z=0.5cm,rotate around y=25, rotate around x=90]
\foreach \i in {1,3,5,7,9,11,13,15,17,19,21}
{
\coordinate (B1) at (\i,0,0);
\coordinate (B2) at (\i+1,0,1);
\coordinate (B3) at (\i+1,5,1);
\coordinate (B4) at (\i,5,0);

\coordinate (A1) at (\i-1,0,1);
\coordinate (A2) at (\i,0,0);
\coordinate (A3) at (\i,5,0);
\coordinate (A4) at (\i-1,5,1);

\draw[fill=black!40] (A1)--(A2)--(A3)--(A4)--cycle;
\draw[fill=black!40] (B1)--(B2)--(B3)--(B4)--cycle;
}

\end{scope}

%
%
%
%

\draw (0,0) ellipse (5cm and 2.5cm);
\draw[fill=black!20] (0,0) ellipse (2.5cm and 1.25cm);
\node[anchor=south] (Bab2) at (0,0) {$B(0,ab)$}; 

\tikzmath{\k=40;}
\foreach \i in {5,...,15}
{
	\tikzmath{\s=sin(360*(2*\i)/\k);\ss=sin(360*(2*\i+1)/\k);\sss=sin(360*2*(\i+1)/\k); \cc=cos(360*(2*\i+1)/\k); \c=cos(360*(2*\i)/\k); \ccc=cos(360*2*(\i+1)/\k);}
	
	\draw[fill=black!40]  (2.5*\c,1.25*\s)--(5*\c,2.5*\s)--(5*\cc, 2.5*\ss+.55)-- (2.5*\cc,1.25*\ss+.55)--cycle;
	\draw[fill=black!40]  (2.5*\cc,1.25*\ss+.55)--(5*\cc,2.5*\ss+.55)--(5*\ccc, 2.5*\sss)-- (2.5*\ccc,1.25*\sss)--cycle;
}

\tikzmath{\k=40;}
\foreach \i in {4,...,-4}
{
	\tikzmath{\s=sin(360*(2*\i)/\k);\ss=sin(360*(2*\i+1)/\k);\sss=sin(360*2*(\i+1)/\k); \cc=cos(360*(2*\i+1)/\k); \c=cos(360*(2*\i)/\k); \ccc=cos(360*2*(\i+1)/\k);}
	
	\draw[fill=black!40]  (2.5*\cc,1.25*\ss+.55)--(5*\cc,2.5*\ss+.55)--(5*\ccc, 2.5*\sss)-- (2.5*\ccc,1.25*\sss)--cycle;
	\draw[fill=black!40]  (2.5*\c,1.25*\s)--(5*\c,2.5*\s)--(5*\cc, 2.5*\ss+.55)-- (2.5*\cc,1.25*\ss+.55)--cycle;
}

\end{tikzpicture}
		\caption{The idea behind the construction. The annulus $\{(r,\theta): a\leq 1\}$ is first unwrapped and then folded and scaled suitably, so that its projection fits the annulus $\{(r,\theta): ab\leq r\leq 1\}$. Then the folded rectangle is wrapped around that annulus as shown in the last figure.}\label{fig:folding}
	\end{figure}
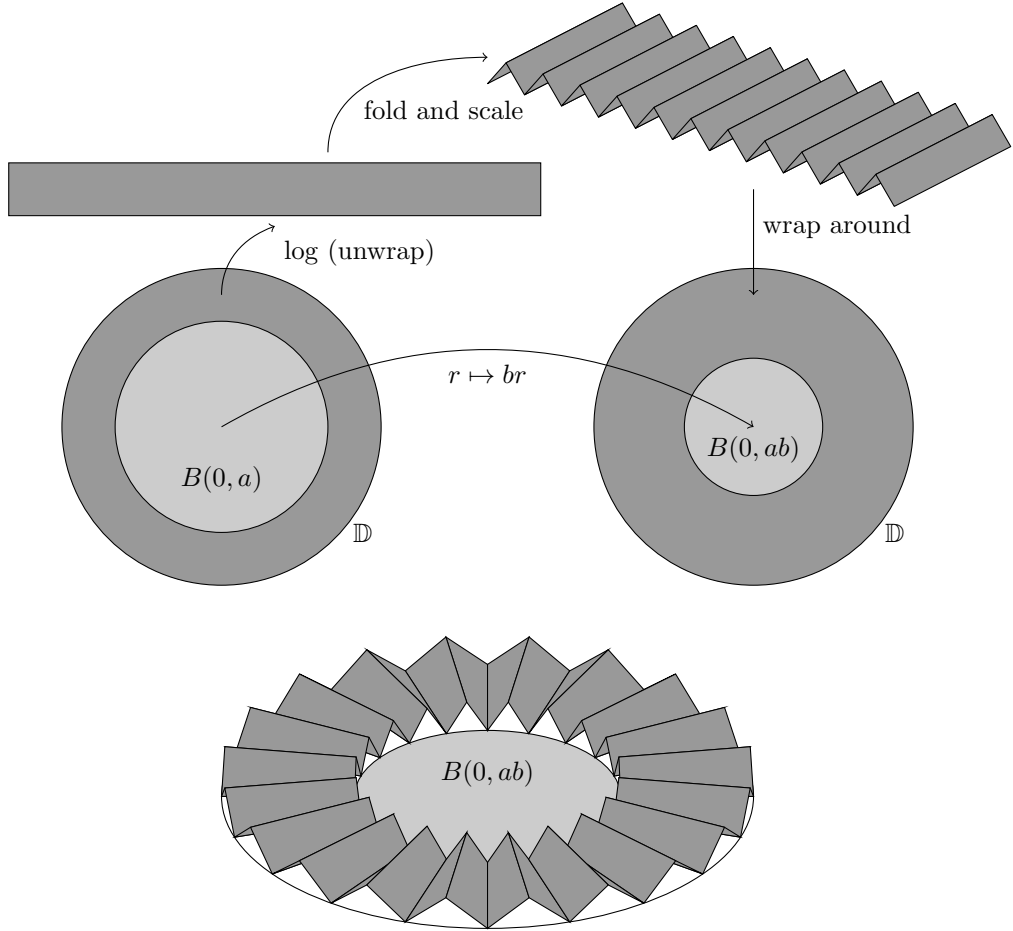
	
	We now give some intuitive explanation why it is possible to construct a mapping with these properties. Let us contrast this with the two-dimensional situation. This type of construction would not be possible if $F_k$ mapped $\br {\mathbb{D}}\times \{0\}$ into the plane instead of $\R^3$. This is because $F_k$ would map the annulus $\{(r,\theta): a \leq r \leq 1\}$ onto the thicker annulus $\{(r,\theta): ab \leq r \leq 1\}$, and it is well-known that any such mapping has dilatation $K \geq 1+\log b/\log a$ (see \cite[Theorem 39.1]{Vais:71}). It would then be impossible to maintain a uniform bound on dilatation in the composition scheme.
	
	By allowing $F_k$ to take values in $\R^3$, we can take advantage of extra space to work in. In rough terms, this is first achieved by cutting the annulus $\{(r,\theta): a \leq r \leq 1\}$ along a radial segment to obtain a rectangle. This rectangle is scaled and folded a certain number of times, and then wrapped back around the vertical axis so that its projection onto $\mathbb{R}^2 \times \{0\}$ is the annulus $\{(r,\theta): r <ab\}$. The parameter $k$ corresponds to the number of folds that are performed.  The folding procedure is a conformal map except along the edges of folding, while the wrapping procedure is close to conformal. By folding the rectangle more times, we may achieve that the wrapping procedure is arbitrarily close to conformal. See Figure \ref{fig:folding} for a sketch of this construction. 
	
	As just described, the mapping is not continuous on the circles $\{(r,\theta): r = a\}$ and $\{(r,\theta):r=1\}$, where we wish to have the identity map. Also, the wrapping procedure must be done more delicately than indicated in this overview to achieve almost conformality. These difficulties will be dealt with in the actual construction. In fact, we will give a precise formula for $F_k$ that is inspired by this folding idea; see Figure \ref{fig:map_f} for an illustration of the actual image of $\br \D\times\{0\}$ under our map $F_k$.

	So far, we have described the mapping on the disk $\br{\mathbb{D}} \times \{0\}$. The map can be extended to a quasiconformal homeomorphism of $T$ by using an orthogonal extension when close to $\br{\mathbb{D}} \times \{0\}$, away from the edges of folding, and then interpolating with the identity map on the boundary $\partial T$. An important point is that the maximum dilatation of this mapping does not depend on the number of the edges where the folding occurs. 
	
	In Section \ref{sec:inverse_absolute_continuity}, we verify the property of inverse absolute continuity for Theorem \ref{thm:main1}. The strategy is to use a probabilistic argument to show that the pointwise Lipschitz constant of the mapping at $x$ is zero for almost every point $x \in A$. We also give a modified argument that yields Theorem \ref{thm:main2}. 
	
	Several auxiliary lemmas from differential geometry and quasiconformal mapping theory are required in the proof and will be quoted when needed. We also use the following two facts about quasiconformal mappings. First, if $f_1, f_2\colon \mathbb{R}^n \rightarrow \mathbb{R}^n$ are, respectively, $K_1$- and $K_2$-quasiconformal, then $f_2 \circ f_1$ is $(K_1K_2)$-quasiconformal. This follows immediately from the equivalence of the analytic definition of quasiconformality with the geometric definition based on modulus of curve families. See, for instance, Theorem 34.6 in V\"ais\"al\"a's notes \cite{Vais:71}.
	
	The second fact is a theorem of V\"ais\"al\"a (\cite[Theorem 35.1]{Vais:71}) that quasiconformality need only be verified up to an exceptional set.  Here is a precise statement. We have included an analogous statement for Lipschitz mappings, which can be proved similarly.  Recall that, given $L >0$, a mapping $f\colon (X,d_X) \rightarrow (Y,d_Y)$ between metric spaces is {\it $L$-Lipschitz} if $d_Y(f(x),f(y)) \leq L d_X(x,y)$ for all $x,y\in X$.
	
	\begin{thm}[\cite{Vais:71}, Theorem 35.1] \label{thm:exceptional_set}
		Let $U,V\subset \R^n$ be open sets and let $E \subset U$ be a closed subset of $\sigma$-finite Hausdorff $(n-1)$-measure.
		\begin{enumerate}[\upshape{(i)}]
			\item  Suppose that $f\colon U \rightarrow V$ is a homeomorphism with the property that for all $x \notin E$, $f$ is $K$-quasiconformal when restricted to some neighborhood of $x$ (for some $K \geq 1$ independent of $x$). Then $f$ is $K$-quasiconformal. 
			\item[\upshape(ii)]  Suppose that $f\colon U \rightarrow V$ is a homeomorphism with the property that for all $x\notin E$, $f$ is $L$-Lipschitz when restricted to some neighborhood of $x$ (for some $L \geq 1$ independent of $x$). Then $f$ is locally $L$-Lipschitz on $U$. Furthermore, if $U$ is convex, then $f$ is $L$-Lipschitz on $U$.
		\end{enumerate}
		
	\end{thm}

	Throughout the paper, we use the notation $\simeq$ to denote comparability. That is, given quantities $x,y \geq 0$ depending on certain parameters, we write $x \simeq y$ if there exists a constant $C>0$ (independent of those parameters) such that $C^{-1}x \leq y \leq Cx$. Similarly, we write $x \lesssim y$ if $x \leq Cy$.  Also, we write $x \ll y$ to indicate that $x$ is very small relative to $y$. Finally, given some parameter $a$, we write $C=C(a)$ to indicate that the quantity $C$ is dependent only on $a$. In using this notation, we follow a convention that the value of $C$ may change from line to line, provided that it is in principle computable and only depends on $a$.
	
	\subsection*{Acknowledgments:}
	Part of the research was conducted while the first named author was visiting University of Jyv\"askyl\"a. He thanks the faculty and staff of the Department of Mathematics for their hospitality. The authors also thank Pekka Koskela for pointing out the reference \cite{HeiKos:94}, and Kai Rajala for his encouragement throughout this project.
	
	\section{Quasiconformal mappings between cylinders} \label{sec:qc_cylinder}
	
	Let $\mathbb{D}$ be the open unit disk in $\mathbb{R}^2$, and let $T$ be the cylinder $\br{ \mathbb{D}} \times [-2,2]$. The goal of this section is to construct a family of quasiconformal homeomorphisms $F_k\colon T \rightarrow T$, indexed by the parameter $k \in \mathbb{N}$, with the properties described in Proposition \ref{prop:basic_construction} below.
	
	We first establish some notation. Points in $\R^2$ are denoted by $(x,y)$ or in polar coordinates by $(r,\theta)$. We use $z$ to denote points in $\mathbb{R}^{3}$, with $(x,y,t)$ for rectangular coordinates and $(r,\theta,t)$ for cylindrical coordinates. Moreover, if $f$ is a mapping from a subset of $\mathbb{R}^2$ into $\mathbb{R}^3$, we use the notation $\widetilde f$ for its representation in polar coordinates (for the domain) and cylindrical coordinates (for the image). Likewise, if $f$ is a mapping between subsets of $\mathbb{R}^3$, we use $\widetilde{f}$ for its representation with respect to cylindrical coordinates. 
	
	The construction in this section also involves parameters $a,b \in (0,1)$. We specify the conditions these must satisfy in Section \ref{sec:inverse_absolute_continuity} below. For a fixed angle $\theta \in [0, 2\pi)$, let $\ell_\theta = \{re^{i\theta}: 0 \leq r \leq 1\}$. For a fixed radius $r \in (0,1)$, let $S_r = \{re^{i\theta}: 0 \leq \theta \leq 2\pi\}$. Also, for a given $k \in \mathbb{N}$, let 
	$$Z_k =  \partial \D \cup S_a \cup \bigcup_{j=0}^{2k-1} \ell_{\pi j/k}.$$
	
	The following proposition gives a precise statement of our construction. 
	
	\begin{prop} \label{prop:basic_construction}
		There exists, for each $k \in \mathbb{N}$, a quasiconformal mapping $F_k\colon T \rightarrow T$ such that
		\begin{enumerate}[\upshape(a)]
			\item $F_k$ is $K$-quasiconformal, for some $K=K(a,b)$ independent of $k$, \label{item:basic_construction(a)}
			\item $F_k|\partial T$ is the identity map, \label{item:basic_construction(b)}
			\item $(\D \setminus Z_k) \times \{0\}$ has an open neighborhood $V_k$ (in $\R^3$) contained in the interior of $T$ on which $F_k$ is $(1+\eta(k))$-quasi\-conformal, where $\eta(k) \rightarrow 0$ as $k \rightarrow \infty$, \label{item:basic_construction(c)}
			\item {restricted to $(B(0,a) \times \mathbb{R})\cap V_k $, the map $F_k$ is a conformal scaling by the factor $b<1$, that is, the point $z\in \R^3$ is mapped to $bz$,} and \label{item:basic_construction(d)}
			\item $F_k$ is $L$-Lipschitz, for some $L=L(a)$ depending only on $a\in (0,1)$ and independent of $k$. \label{item:basic_construction(e)}
		\end{enumerate}
	\end{prop}
	
	In parts \eqref{item:basic_construction(d)} and \eqref{item:basic_construction(e)}, the notation $B(0,a)$ refers to the ball in $\mathbb{R}^2$. To simplify our notation, we write $F$ in place of $F_k$ for the remainder of this section. 
	
	\begin{rem}
		Note that the domain of a quasiconformal map is, by definition, an open subset of $\R^n$. As a general rule, by specifying a closed set $E$ as the domain of a quasiconformal mapping, it is implied that $F$ can be extended to a quasiconformal homeomorphism between $U$ and $V$, where $U$ is an open neighborhood of $E$ and $V$ is an open neighborhood of $F(E)$. In the case of Proposition \ref{prop:basic_construction}, we can extend $F$ to a quasiconformal homeomorphism of $\mathbb{R}^3$ by defining it as the identity map outside of $T$. 
	\end{rem}
	
	We proceed now with the proof of Proposition \ref{prop:basic_construction}, which will occupy the remainder of this section. Fix a sufficiently large $k \in \mathbb{N}$. Let $D_k = \{(r, \theta): r \leq 1, 0 \leq \theta \leq \pi/k\}$, in cylindrical coordinates, and let $T_k = D_k \times [-2,2]$.
	We will define the mapping $F$ first on $T_k$, and then on all of $T$ using symmetry as we now explain. More specifically, $F$ will map $T_k$ onto itself, taking the sets $\ell_{\pi/k} \times [-2,2]$, $\ell_0 \times [-2,2]$ onto themselves and acting as the identity map on $\partial T_k \setminus (\ell_{\pi/k} \times [-2,2]\cup \ell_0\times [-2,2])$.  
	Once $F$ has been defined in $T_k$ we can then extend $F$ to the set $\{(r, \theta,t) \in T: \pi/k \leq \theta \leq 2\pi/k\}$ by reflection across $\ell_{\pi/k} \times [-2,2]$. Namely, if $\widetilde F(r,\theta,t)=(\mathcal R(r,\theta,t), \Theta(r,\theta,t), \mathcal T(r,\theta,t)) $ is the representation of $F$ in cylindrical coordinates, then we define
	\begin{align*}
	\widetilde F(r,\theta,t)= (\mathcal R(r,2\pi/k-\theta,t), 2\pi/k - \Theta(r, 2\pi/k- \theta,t), \mathcal T(r,2\pi/k-\theta,t))
	\end{align*}
	in cylindrical coordinates, for $\pi/k\leq \theta\leq 2\pi/k$. Finally, we define $F$ on the other wedges $\{(r, \theta): r \leq 1, 2\pi (l-1)/k \leq \theta \leq 2\pi l/k\}\times [-2,2]$, where $1\leq l\leq k$, by 
	\begin{align*}
	\widetilde F(r,\theta,t)= \widetilde F(r, \theta - 2\pi(l-1)/k,t) + (0,2\pi(l-1)/k,0). 
	\end{align*}
	Observe that the definition of $F$ is consistent on the boundary of the wedges, so that $F$ is indeed a homeomorphism of $T$. As previously noted, $F$ will act as the identity map on the boundary $\partial T$. Moreover, Theorem \ref{thm:exceptional_set} implies that $F$ is quasiconformal on $T$ with the same dilatation as on $T_k$ and also Lipschitz on $T$ with the same constant as on $T_k$. It is immediate that Proposition \ref{prop:basic_construction} will follow once we have shown that $F$ satisfies properties \eqref{item:basic_construction(a)}-\eqref{item:basic_construction(e)} when restricted to the set $T_k$.

	What remains then is to construct the map $F$ on the set $T_k$.	We introduce the following five subsets of $T_k$. For simplicity, we drop $k$ from the notation:
	\begin{enumerate}[\upshape(i)]
		\item $S_2 = \{ z \in T_k: 1\leq t \leq 2\}$
		\item $S_1 = \{ z \in T_k: \varphi(x,y) \leq t \leq 1\}$
		\item  $E = \{ z \in T_k: |t| \leq  \varphi(x,y)\}$
		\item  $S_1' = \{ z \in T_k:  -1\leq t \leq -\varphi(x,y)\}$
		\item  $S_2' = \{ z \in T_k: -2 \leq t\leq -1 \}$.
	\end{enumerate}
	In the above notation, $\varphi\colon \R^2 \rightarrow [0,1]$ is a certain non-negative smooth function with $\partial D_k=\varphi^{-1}(0)$ and $|\varphi|\ll 1$. We will later choose this function carefully in Section \ref{sec:S_1} and it will depend on $k$; see also Lemma \ref{lemma:bump function} and the discussion following it.
	\subsection{Definition of \texorpdfstring{$F$}{F} on the set \texorpdfstring{$D_k \times \{0\}$}{Dk x {0}}} 
	We first define an auxiliary function $g\colon [0,\infty) \rightarrow [0,\infty)$ specifying how $F|D_k \times \{0\}$ affects the radius of a given point. Its definition depends on the parameters $a,b \in (0,1)$ from Proposition \ref{prop:basic_construction} above. 
	
	\begin{lemm}\label{lemma:g}	
		For each $a,b\in (0,1)$ there exists a smooth, strictly increasing function $g_{a,b}=g\colon [0,\infty) \rightarrow [0,\infty)$  such that 
		\begin{enumerate}[\upshape(i)]
			\item $g(r) = br$ for all $r \in [0,a]$ and $g(r) = r$ for $r\geq 1$, \label{item:g(a)}
			\item $b\leq g'(r) \leq L$ for some constant $L=L(a)>0$ for $r\in [0,\infty)$,\label{item:g(b)}
			\item $g'(r) \geq g(r)/r$ for $r\in [0,\infty)$, and \label{item:g(c)}
			\item the function $h(r)= \sqrt{g'(r)^2- g(r)^2/r^2}$ is smooth for $r\in [0,\infty)$. \label{item:g(d)}
		\end{enumerate}
	\end{lemm}
	
	\begin{proof}
		The function $g$ can be defined as follows. Let $\rho\colon \R\to [0,1]$ be defined by 
		$$\rho(x)= \frac{\sigma(x)}{\sigma(x)+\sigma(1-x)},$$
		where $\sigma(x)= e^{-1/x} \x_{(0,\infty)}(x)$. Observe that $\rho$ is an increasing smooth function satisfying $\rho(x)=0$ for all $x\leq 0$, $\rho(x)=1$ for all $x\geq 1$, and $\rho'(x)>0$ for all $x \in (0,1)$.  Then let
		\begin{align*}
		g(r) =  (1-b) \rho \left( \frac{r-a}{1-a} \right) r+br.
		\end{align*}
		It is immediate that $br\leq g(r)\leq r$ and that $g$ has property \eqref{item:g(a)} by construction. Next, a computation yields 
		\begin{align}\label{lemma:g:derivative}
		g'(r)= \frac{1-b}{1-a} \rho'\left( \frac{r-a}{1-a} \right)r  + \frac{g(r)}{r}>0.
		\end{align}
		Since $\rho'\geq 0$, this establishes property \eqref{item:g(c)}. Since $b \leq g(r)/r$, this also verifies the first inequality in \eqref{item:g(b)}. The other inequality in \eqref{item:g(b)} follows from
		\begin{align*}
		g'(r)\leq \frac{1-b}{1-a}\|\rho'\|_\infty \x_{(a,1)}(r)+1 \leq \frac{1}{1-a} \|\rho'\|_\infty +1 \eqqcolon L,
		\end{align*}
		using the fact that $a, b<1$.  
		
		Finally, we verify \eqref{item:g(d)}. Note that the function $h$ vanishes for all  $r\in [0,a]$, so we may assume that $r \geq a$.  Write $h$ in the form 
		\begin{align*}
		h(r)= \sqrt{g'(r)-g(r)/r} \cdot \sqrt{g'(r)+g(r)/r},
		\end{align*}
		where the function $g'(r)+g(r)/r$ is non-vanishing. It follows that $\sqrt{g'(r)+g(r)/r}$ is smooth.  Next, we wish to show that
		$$\sqrt{g'(r)-g(r)/r} = \sqrt{\frac{1-b}{1-a}\rho'\left(\frac{r-a}{1-a}\right)r}$$ 
		is smooth.  Since $r \geq a$, the function $\sqrt{r}$ is smooth and it suffices to verify the smoothness of $\sqrt{\rho' \left( \frac{r-a}{1-a}\right)}$ at $r=a$ and at $r=1$.
		
		By a change of coordinates, this reduces to showing that $\sqrt{\rho'(x)}$ is smooth at $x=0$ and at $x=1$. Observe that $\rho(x)+\rho(1-x)=1$, so $\rho'(x)=-\rho'(1-x)$. Hence  we need only consider the case $x=0$.
		
		Note that $\rho(x)= \rho_0(x)e^{-1/x}\x_{(0,\infty)}(x)$,  where the function $$\rho_0(x) = (e^{-1/x}\x_{(0,\infty)}(x) + e^{-1/(1-x)}\x_{(0,\infty)}(1-x))^{-1}$$ is smooth. Taking derivatives gives $\rho'(x)=(\rho_0'(x)+\rho_0(x)/x^2)e^{-1/x}\x_{(0,\infty)}(x)$. The function $ \rho_0'(x)+\rho_0(x)/x^2$ is non-negative (since $\rho'\geq 0$) and each of its derivatives is $O(1+1/x+\dots+1/x^m)$ for some $m\in \N$ on an interval $(0,\delta)$, by the smoothness of $\rho_0$ in neighborhood of $0$. The same is true for the function $\omega_0(x) = \sqrt{ \rho_0'(x)+\rho_0(x)/x^2}$. 
		
		Let $\psi(x) =\sqrt{\rho'(x)}= e^{-1/(2x)}{\omega_0(x)}\x_{(0,\infty)}$. This function is continuous at $0$. Inductively one can show that  the $q$-th order derivatives $\psi^{(q)}$ satisfy $\psi^{(q)}(x)= e^{-1/(2x)} \omega_q(x)\x_{(0,\infty)}$ for some smooth function $\omega_q(x)= O(1+1/x+\dots+1/x^{m})$ near $0$,  for some $m\in \N$ depending on $q$. This shows that $\psi^{(q)}$ is continuous at $0$ for all $q\in \N\cup \{0\}$. 
	\end{proof}

	We now define $F|D_k \times \{0\}$ to be the mapping $G\colon D_k \rightarrow \mathbb{R}^3$ given by 
	\begin{align*}
	{G}(x,y) = \left( \frac{g(r) x}{r} , \frac{g(r) y}{r}, r  h(r)\theta\right)
	\end{align*}
	in rectangular coordinates, where $r=\sqrt{x^2+y^2}$ and $\theta=\arctan(y/x)\in (-\pi/2,\pi/2)$. Note that $G$ extends smoothly to a neighborhood of $ D_k$ because $h(r)=0$  for all $r<a$. We will denote by $\Omega_k$ the open $1/k$-neighborhood of $D_k$ in $\R^2$. 
	
	\begin{figure}
		\centering
		\includegraphics[scale=.5]{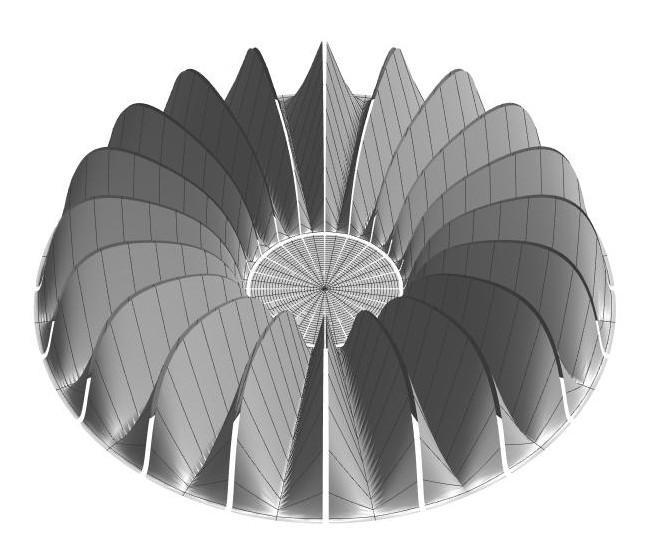}
		\caption{The parametric surface defined by $G$ after periodic extension.}\label{fig:map_f}
	\end{figure}		
	
	\begin{lemm}\label{lemma:properties of f}
		The function $G$ has the following properties:
		\begin{enumerate}[\upshape(i)]
			\item $G$ is a smooth embedding of $\Omega_k$ into $\R^3$, for large enough $k\in \N$. \label{item:properties_f(a)}
			\item $G$ is $(1+\eta(k))$-quasiconformal in $\Omega_k$, with $\eta(k)\to 0$ as $k\to\infty$. More specifically, the tangent vectors $G_x$ and $G_y$ have magnitude $g'(r)+o(1)$ and their angle is $\pi/2 +o(1)$ as $k\to\infty$. \label{item:properties_f(b)}
			\item The partial derivatives of $G$ are bounded in $\Omega_k$ by some constant $L=L(a)>0$, independent of $k$, provided that $k\in \N$ is sufficiently large. \label{item:properties_f(c)}
			\item The set $\{G(x,y): (x,y)\in  \Omega_k\}\subset \R^3$ is the graph of a function of two variables for all $k\in \N$. Moreover, for sufficiently large $k\in \N$ the map $f$ takes $\ell_{\pi/k}$ and $\ell_0$ into $\ell_{\pi/k}\times [-1/2,1/2]$ and $\ell_0\times [-1/2,1/2]$, respectively. Finally, it maps the arc $\{(1,\theta): 0\leq \theta\leq \pi/k\}$ onto the arc $\{(1,\theta,0):0\leq \theta\leq \pi/k\}$ acting as the identity map. In particular, $G(D_k)$ is contained in $T_k$. \label{item:properties_f(d)}
		\end{enumerate}
	\end{lemm}

	\begin{rem}Here the quasiconformality of a smooth embedding $G$ is to be understood as follows. The differential matrix $DG(x,y)$ is required to be a full-rank $3\times 2$ matrix, so the image of $DG(x,y)$ is a plane in $\R^3$. Restricting to that plane (or composing with a rotation of $\R^3$), we may regard $DG(x,y)$ as a $2\times 2$ non-singular matrix $A$. We say that $G$ is $K$-quasiconformal at $(x,y)$ if $\|A\|^2\leq K \det(A)$, that is, if \textit{the matrix $A$ is $K$-quasiconformal}. For example, if 
	$$D{G}(x,y) = \begin{pmatrix} a & 0 \\ 0 & b \\ 0 & c \end{pmatrix},$$
	then the image of $DG(x,y)$ is the plane spanned by the vectors $(a,0,0)$ and $(0,b,c)$. In this case, the matrix $DG(x,y)$ is conformal if it has full rank and $a^2=b^2+c^2$.
	\end{rem}
	
	We need a criterion to decide whether a map is $(1+\eta)$-quasiconformal for a small $\eta>0$. The following lemma provides such a criterion:
	\begin{lemm}\label{lemma:matrix qc}
		Let $A\in M_n(\R)$ be a non-singular matrix with $\det(A)>0$, and consider a positive number $R>0$. Denote the columns of $A$ by $A_1,\dots,A_n$. Then for each $\eta>0$, there exists $\delta>0$ such that the following statement holds: if $||A_i|-R|<\delta$ and $|\langle A_i,A_j\rangle| < \delta$ for all $i,j\in \{1,\dots,n\}$, $j\neq i$, then $A$ is $(1+\eta)$-quasiconformal, i.e., $\|A\|^n\leq (1+\eta) \det(A)$.
	\end{lemm} 
	This is to say that if all the columns of a matrix have almost the same magnitude and they are almost perpendicular to each other, then the matrix is close to being conformal, that is, an isometry up to scaling. This conclusion applies also to $3\times 2$ matrices, in view of the preceding remark. We provide a sketch of the elementary proof.
	\begin{proof}
		By scaling, we may assume that $R=1$. Suppose that the statement fails, so we can find a sequence of matrices $A(k)$, $k\in \N$, converging to a matrix $A$ with $|A_i|=1$, $\langle A_i,A_j\rangle =0$ for $i\neq j$, and $\|A(k)\|^n > (1+\eta)\det(A(k))$ for all $k\in \N$ and for some $\eta>0$. Passing to the limit, we have $\|A\|^n \geq (1+\eta) \det(A)$. However, the matrix $A$ is orthogonal, which is a contradiction. 
	\end{proof}

	\begin{proof}[Proof of Lemma \ref{lemma:properties of f}]
		Suppose for the moment that $G$ is an embedding of $\Omega_k$ into $\R^3$. Observe that $G$ is conformal outside the annulus $a\leq r\leq 1$ and that the definition of $G$ does not depend on $k$; it is only the domain that changes in $k$. The mapping $G$ has differential matrix
		\begin{align*}
		D{G}(x,y) = \begin{pmatrix} \frac{y^2g(r)}{r^3} + \frac{x^2g'(r)}{r^2} & xy \left( -\frac{g(r)}{r^3} + \frac{g'(r)}{r^2}\right) \\ xy \left( -\frac{g(r)}{r^3} + \frac{g'(r)}{r^2}\right) & \frac{x^2g(r)}{r^3} + \frac{y^2g'(r)}{r^2} \\ (-y+x\theta)\frac{h(r)}{r} + x \theta h'(r) & (x+y\theta)\frac{h(r)}{r} + x \theta h'(r) \end{pmatrix}.
		\end{align*}
		If $y=0$, then $x=r$ and $\theta = 0$. In this case, the differential matrix simplifies to
		\begin{align}\label{Df}
		D{G}(x,0) = \begin{pmatrix} g'(x) & 0 \\ 0 & \frac{g(x)}{x} \\ 0 & h(x) \end{pmatrix}.
		\end{align}
		Since $g'(x)^2= g(x)^2/x^2+h(x)^2$, it follows by the discussion after the statement of the lemma that ${G}$ is conformal along the positive $x$-axis. Hence by smoothness $G$ is close to conformal  when $y$ is sufficiently small. This is to say that when $y$ is near $0$ the tangent vectors $G_x$ and $G_y$ have almost they same length, very close to $g'(r)$, and they are almost perpendicular to each other. Thus, the conclusion from this calculation is that for all $\eta >0$, one can pick $k$ large enough so that $G$ is $(1+ \eta)$-quasiconformal in $\Omega_k$. This proves \eqref{item:properties_f(b)}.
		
		In order to show that $G$ is an embedding, we first observe that it is an immersion at points $(x,0)$, since its differential matrix has full rank by the previous computation. Indeed, $g'\simeq 1$ by Lemma \ref{lemma:g}\eqref{item:g(b)}. As soon as $k$ is sufficiently large, the region $\Omega_k$ is very close to the positive real axis, so we obtain by smoothness that $DG$ also has full rank there. To see that $G$ is a homeomorphism onto its image, we observe that $G$ does not change the polar angle, hence if $G(x,y)=G(x',y')$, then $\theta=\theta'$. Since $g$ is strictly increasing, we also obtain $r=r'$, which implies that $(x,y)=(x',y')$. Therefore, $G$ is injective in $\Omega_k$, that is, \eqref{item:properties_f(a)} holds. 
		
		The same argument shows that $G(\Omega_k)$ is the graph of a function. Namely, no vertical line intersects $\{G(x,y): (x,y)\in \Omega_k\}$ at two points. By the definition of $G$ and the fact that it does not change the polar angle, it is clear that it takes $\ell_{\pi/k}$ and $\ell_0$ into $\ell_{\pi/k}\times \R$ and $\ell_0\times \R$, respectively. If $k$ is sufficiently large, then the $t$-component $rh(r)\theta$ of $G$ is small, so this justifies the claims about the image of $\ell_{\pi/k}$ and $\ell_0$. Finally, since $h(1)=0$ and $g(1)=1$ (see Lemma \ref{lemma:g}), the map $G$ acts as the identity on the arc $\{(1,\theta): 0\leq \theta\leq \pi/k\}$. This completes the proof of \eqref{item:properties_f(d)}.
		
		For the derivative bounds in \eqref{item:properties_f(c)}, note that for each $\delta>0$ we can choose a sufficiently large $k$ such that $|G_x|\leq g'(r)(1+\delta)$ and $|G_y|\leq g'(r)(1+\delta)$ in $\Omega_k$; see \eqref{Df}. Hence, we obtain the upper bound $g'(r)(1+\delta)\leq  2L$, provided that $\delta$ is sufficiently small, where $L$ is as in Lemma \ref{lemma:g}\eqref{item:g(b)}.
	\end{proof}

	\subsection{Definition of \texorpdfstring{$F$}{F} on the set \texorpdfstring{$E$}{E}}
	
	Consider the unit normal vector $\mathbf{N}= G_x\times G_y/ |G_x\times G_y|$ defined on $\Omega_k$. Let 
	$$F(x,y,t)=G(x,y)+t g'(r)\cdot \mathbf N(x,y),$$
	where $(x,y,t)\in  \Omega_k \times \R$. Later we will restrict the domain of $F$. We record some immediate observations:
	
	\begin{lemm}\label{lemma:E:Lipschitz}
		For sufficiently large $k\in \N$ the following statements are true.
		\begin{enumerate}
			\item[\upshape(d$'$)] Restricted to $(\Omega_k\cap B(0,a))\times \R$, the map $F$ is a conformal scaling by the factor $b$.
			\item[\upshape(e$'$)] For sufficiently small $\varepsilon>0$ the map $F$ is locally Lipschitz on $D_k\times [-\varepsilon,\varepsilon]$ with Lipschitz constant $L=L(a)>0$, independent of $k$.
		\end{enumerate}
	\end{lemm}
	
	These properties will yield the required properties \eqref{item:basic_construction(d)} and  \eqref{item:basic_construction(e)} in Proposition \ref{prop:basic_construction}, after we restrict the domain of $F$ properly. 
	
	\begin{proof}
		If $(x,y)\in B(0,a)$, then $g'(r)=b=g(r)/r$ and $G(x,y)=b(x,y,0)$. Therefore, $\mathbf N(x,y)=(0,0,1)$ and $F(x,y,t)=b(x,y,0)+tb(0,0,1)=b(x,y,t)$. For the Lipschitz condition note that when $t=0$ the vectors $F_x(x,y,0)=G_x(x,y)$, $F_y(x,y,0)=G_y(x,y)$, and $F_t(x,y,0)=g'(r)\mathbf N(x,y)$ have magnitude bounded by some constant $L>0$, independent of $k$, by Lemma \ref{lemma:properties of f}\eqref{item:properties_f(c)} and Lemma \ref{lemma:g}\eqref{item:g(b)}. By smoothness, there exists an $\varepsilon>0$ such that the partial derivatives of $F$ are bounded by $2L$ in $ D_k\times [-\varepsilon,\varepsilon]$. This implies that $F$ is locally Lipschitz, as claimed.
	\end{proof}
	
	Next, we claim the following:
	
	\begin{lemm}\label{lemma:E:qc}
		\upshape(c$'$) For each $k\in \N$ there exists $\varepsilon>0$ such that $F$ is a  $(1+\eta(k))$-quasi\-conformal embedding of $D_k\times [-\varepsilon,\varepsilon]$ into $\R^3$, with $\eta(k)\to 0$ as $k\to\infty$. 
	\end{lemm}
	
	This lemma yields the required property \eqref{item:basic_construction(c)} in Proposition \ref{prop:basic_construction}, after we restrict the domain of $F$ properly. 
	
	\begin{proof}
		To show that $F$ is an embedding we use the following tubular-neighborhood lemma:
		\begin{lemm}
			Let $\Omega \subset \R^n$ be an open set and $q\colon \Omega \to \R^{n+1}$ be a smooth embedding. Suppose that $\mathbf v\colon \Omega \to \R^{n+1}$ is a smooth vector field such that the matrix $(Dq(w),v(w))\in M_{n+1}(\R)$ is non-singular at each point $w\in \Omega$.  Then for each compact set $ D \subset \Omega$ there exists $\varepsilon>0$ such that the map 
			\begin{align*}
			Q(w,t)= q(w)+ tv(w)
			\end{align*}
			is an embedding of $D \times [-\varepsilon,\varepsilon]$ into $\R^{n+1}$.
		\end{lemm}
		We omit the proof of that lemma but it can be done along the lines of the proof of the existence of tubular neighborhoods of embedded submanifolds of $\R^n$; see for instance \cite[Theorem 6.24]{Lee:13}.

		Note that $f$ is an embedding of a neighborhood $\Omega_k$ of $ D_k$ into $\R^3$, by Lemma \ref{lemma:properties of f}\eqref{item:properties_f(a)}. Moreover, $(DG(x,0),\mathbf N(x,0))$ is non-singular, since the tangent vectors $G_x(x,0)$, $G_y(x,0)$ have magnitude $g' \simeq 1$ (see Lemma \ref{lemma:properties of f}\eqref{item:properties_f(b)} and Lemma \ref{lemma:g}\eqref{item:g(b)}), and $\mathbf N(x,0)$ is the unit normal vector. By smoothness, we conclude that the matrix $(DG(x,y),\mathbf N(x,y))$ is non-singular when $y$ is small, and hence when $(x,y)$ lies in the set $\Omega_k$ for sufficiently large $k$. Thus the preceding lemma applies to yield that $F$ is an embedding of $ D_k\times [-\varepsilon,\varepsilon]$ into $\R^3$. Here $\varepsilon$ depends on $k$.
		
		By making $\varepsilon$ even smaller, we may have that $F$ is close to being conformal in $D_k\times [-\varepsilon,\varepsilon]$. Indeed, when $t=0$, the vectors $F_x(x,y,0)=G_x(x,y)$, $F_y(x,y,0)=G_y(x,y)$,  and $F_t(x,y,0)=g'(r)\mathbf N(x,y)$ have almost the same magnitude and are almost perpendicular to each other, as $k\to \infty$; see Lemma \ref{lemma:properties of f}\eqref{item:properties_f(b)}. Moreover, the matrix $(F_x,F_y,F_t)$ has positive determinant at points $(x,y,0)$, since $F_t$ is parallel to $G_x\times G_y$. By smoothness, these properties hold also for points $(x,y,t)$ for all sufficiently small $t$, say for $|t|\leq \varepsilon$.  Using Lemma \ref{lemma:matrix qc} we obtain the conclusion that $F$ is close to conformal in $D_k\times [-\varepsilon,\varepsilon]$.
	\end{proof}
	
	Next, we wish to restrict the domain of $F$ to the set $E = \{ z \in T_k: |t| \leq  \varphi(x,y)\}$, for some suitably chosen smooth function $\varphi$. We make this choice using the next lemma:
	
	\begin{lemm}[Theorem 2.29, \cite{Lee:13}]\label{lemma:bump function}
		For each closed set $A\subset \R^n$ there exists a non-negative smooth function $\psi\colon\R^n\to \R$ with $A=\psi^{-1}(0)$.
	\end{lemm}
	
	Let $k\in \N$ be large enough and $\varepsilon=\varepsilon_k>0$ be so small that the conclusions of Lemma \ref{lemma:E:Lipschitz} and Lemma \ref{lemma:E:qc} hold in $D_k\times [-\varepsilon,\varepsilon]$. We consider a smooth function $\psi=\psi_k\colon \R^2 \to [0,\infty)$ as above, so that $0\leq \psi\leq \varepsilon$ and $\psi^{-1}(0)=\partial D_k$. Then the above conclusions for the regularity of $F$ hold in the set $E =  \{z\in T_k: |t|\leq s\cdot\psi (x,y)\} \subset  D_k\times [-\varepsilon,\varepsilon]$, where $s\in (0,1)$ is a small positive constant that will be chosen later. Note that $E$ depends on $k$ and $s$, and that it contains an open neighborhood of $\inter(D_k)\times \{0\}$, since $\psi>0$ on $\inter(D_k)$. This is crucial for part \eqref{item:basic_construction(c)} of Proposition \ref{prop:basic_construction}. 
	
	Summarizing, we have shown that $F|E$ satisfies the conditions \eqref{item:basic_construction(c)},\eqref{item:basic_construction(d)}, and \eqref{item:basic_construction(e)} from Proposition \ref{prop:basic_construction}. What remains to do is to define $F$ on the other pieces $S_i,S_i'$, $i=1,2$, of the partition of $T_k$ in such a way that $F$ maps $T_k$ onto itself quasiconformally with a uniform  bound on dilatation and a uniform Lipschitz bound (as in \eqref{item:basic_construction(e)}), taking the sets $\ell_{\pi/k} \times [-2,2]$, $\ell_0 \times [-2,2]$ onto themselves and acting as the identity map on $\partial T_k \setminus (\ell_{\pi/k} \times [-2,2]\cup \ell_0\times [-2,2])$; recall the discussion after Proposition \ref{prop:basic_construction}. The definition of $F$ on the other pieces of $T_k$ will be given in the next Sections \ref{sec:S_1} and \ref{sec:S_2}. We remark that we have not proved yet that $F(E)\subset T_k$; it could be the case that $F(E)$ falls out of the set $T_k$. We will ensure that this is not the case in Section \ref{sec:S_1}, by choosing a sufficiently small $s>0$.
	
	Before proceeding, we will discuss here some orientation matters related to the image $F(E)$. The set $E$ is homeomorphic to a closed ball in $\R^3$, and therefore, so is its image $F(E)$. We write $E=E^+\cup E^-$, where $E^+= \{z\in E: 0\leq t\leq s\psi(x,y)\}$ and $E^-= \{z\in E: -s\psi(x,y)\leq t\leq 0\}$. Since $E^+$ and $E^-$ are also both homeomorphic to closed balls, the same is true for their images under $F$. Recall that the set $F(D_k\times \{0\})=\{G(x,y):(x,y)\in D_k\}$ is the graph of a function by Lemma \ref{lemma:properties of f}\eqref{item:properties_f(d)}. We claim that $F(E^+)$ lies ``above" the graph $F(D_k\times\{0\})$ and $F(E^-)$ lies ``below" this graph, in the following sense:
	
	\begin{lemm}\label{lemma:orientation of F(E)}
		Let $\mathcal L$ be a vertical line that meets the set $F(\inter(D_k)\times \{0\})$ at a point $z_0$. Suppose that the line $\mathcal L$ is parametrized as it runs from $+\infty$ to $z_0$. Then $\mathcal L$ meets $F(E^+)$ before hitting $z_0$.
	\end{lemm}
	
	Since the set $F(D_k\times \{0\})=\{G(x,y):(x,y)\in D_k\}$ is the graph of a function, the point $z_0$ is unique.
	
	\begin{proof}
		The point $z_0$ is an interior point of the topological ball $F(E)$. Therefore, there exists some open vertical line segment $(z_0,z_1)\subset \mathcal L$ that is contained in the interior of $F(E)$, where $z_1$ lies ``above" $z_0$. This segment does not intersect $F(D_k\times \{0\})$, so it has to be contained in the interior of either the topological ball $F(E^+)$ or the ball $F(E^-)$. We claim that this segment is contained in $F(E^+)$ and this will complete the proof.
		
		Suppose that $z_0=F(x_0,y_0,0)$ for some $(x_0,y_0)\in \inter(D_k)$. We consider the tangent plane of the parametric surface $F(D_k\times \{0\})$ at $(x_0,y_0)$. By direct computation one sees that the normal vector $\mathbf{N}=G_x\times G_y/|G_x\times G_y|$ has positive $t$-coordinate at $(x_0,y_0)$, as soon as $k$ is sufficiently large; see \eqref{Df}. This implies that the segment $(z_0,z_1)\subset \mathcal L$ and points $F(x_0,y_0,t)=z_0+tg'(r_0)\mathbf{N}(x_0,y_0)\in F(E^+)$ for small $t>0$ lie in the same half-space  determined  by the tangent plane. Now we can derive that $(z_0,z_1)$ is contained in $F(E^+)$ by the following elementary fact: 
		\begin{lemm}
			Suppose that $q\colon U\to \R^{n+1}$ is a smooth embedding, where $U$ is an open subset of $\R^{n}$. Then for each $z\in U$ and for each solid double cone $C$ with vertex at $q(z)$ and axis parallel to the normal vector of the surface $q(U)$ at $q(z)$ there exists a small ball $B(q(z),r)$ such that the intersection $B(q(z),r)\cap q(U)\cap C$ contains only $q(z)$.
		\end{lemm}
		In our case, we can connect a point $F(x_0,y_0,t)$ for small enough $t$ to a point of $(z_0,z_1)\subset \mathcal L$ with a straight line segment lying in an appropriate cone, without hitting the set $F(D_k\times \{0\})$. This shows that the segment $(z_0,z_1)$ and $(x_0,y_0,t)$ both lie in the same component of $F(E)\setminus F(D_k\times \{0\})$ and thus they both lie in $F(E^+)$, as desired. 
	\end{proof}

	\subsection{Definition of \texorpdfstring{$F$}{F} on \texorpdfstring{$S_1$}{S1} and \texorpdfstring{$S_1'$}{S1'}}\label{sec:S_1}
	
	Given $s \in [0,1]$, we define $S_1(s)= \{z\in T_k: s\psi(x,y)\leq t\leq 1\}$. For a point $z=(x,y,t) \in S_1(s)$ and for $s\in [0,1]$, define a map $F^s(z)$ by mapping the vertical line from $(x,y, s\psi(x,y))$ to $(x,y,1)$ linearly onto the vertical line from $F(x,y,s\psi(x,y))$ to $(F_1(x,y,s\psi(x,y)), F_2(x,y,s\psi(x,y)), 1)$; here  we have written $F$ in coordinates as $F=(F_1,F_2,F_3)$. More explicitly,
	\begin{align*}
	F^s(x,y,t) = F(x,y,s\psi(x,y)) + \frac{t-s\psi(x,y)}{1-s\psi(x,y)}\left(1 - F_3(x,y, s\psi(x,y))\right)\e .
	\end{align*} 
	Recall at this point that $s\psi\leq \psi\leq  \varepsilon \ll 1$, and that $\varepsilon$ and $\psi$ depend on $k$.
	
	\begin{lemm}\label{lemma:S_1}
		For all sufficiently small $s>0$ the map $F^s|S_1(s)$ is a $K$-quasi\-conformal embedding, where $K=K(a,b)$ is independent of $k$, and locally $L$-Lipschitz, where $L=L(a)>0$ is independent of $k$. More precisely, for each sufficiently large $k\in \N$ there exists $\delta>0$ such that for all $s\in (0,\delta]$ the map $F^s|S_1(s)$ is a locally $L$-Lipschitz $K$-quasiconformal embedding, with $L=L(a)$ and $K=K(a,b)$ independent of $k$.
	\end{lemm}
	
	Note here that $S_1(s)$ depends on $k$. For the proof of Proposition \ref{prop:basic_construction} the desired map $F$ will be chosen to be one of the maps $F^s$ for a sufficiently small $s>0$.
	
	\begin{proof}
		First, we write $F^s$ as the composition of two maps, which are linear on vertical lines: 
		\begin{enumerate}[\upshape(i)]
			\item a map $\alpha^s$ from $S_1(s)$ onto $\{z\in T_k: 0\leq t\leq 1\}= D_k\times [0,1]$, and
			\item a map $\beta^s$ from $\{z\in T_k: 0\leq t\leq 1\}$ onto $\{z\in T_k : F_3(x,y,s\psi(x,y)) \leq t\leq 1\}$. 
		\end{enumerate}
		It suffices to show that for all sufficiently small $s\in (0,1]$ the maps $\alpha^s$ and $\beta^s$ are locally Lipschitz (with uniform constants) quasiconformal homeomorphisms, with a bound on dilatation independent of $k$. To prove this, we use the following stability lemma:
		
		\begin{lemm}\label{lemma:stability}
			Let $\Omega\subset \R^n$ be an open set. Suppose that $\zeta_s\colon \Omega\to \R^n$, $s\in [0,1]$, is a family of  smooth maps depending smoothly on $s\in [0,1]$ such that $\zeta_0$ is a locally $L$-Lipschitz $K$-quasiconformal embedding. Then for each compact set $D\subset \Omega$ and for each $\eta>0$ there exists $\delta>0$ such that $\zeta_s| D$ is a locally $(L+\eta)$-Lipschitz $(K+\eta)$-quasiconformal  embedding for all $s\in [0,\delta]$.
		\end{lemm}
		
		The stability of embeddings of compact manifolds is well-known and a proof can be found in \cite[pp.\ 35--37]{GuiPol:10}. The Lipschitz and quasiconformal stability follow from the fact that the differential $D\zeta_s$ as well as its Jacobian and its operator norm depend continuously on $s$. 
		
		The map $(\alpha^s)^{-1}$ is defined on $\{z\in T_k:0\leq t\leq 1\}$ as 
		\begin{align*}
		(\alpha^s)^{-1}(x,y,t)= (x,y, (1-t)s\psi(x,y)+ t)=\id(x,y,t) +(0,0,(1-t)s\psi(x,y)).
		\end{align*}
		Since $\alpha^0$ is the identity, by the lemma (applied to a neighborhood of $\{z\in T_k:0\leq t\leq 1\}$) we conclude that $\alpha^s$ is a locally $2$-Lipschitz $K$-quasiconformal embedding with $K$ independent of $k$, for all sufficiently small $s>0$.
		
		The map $\beta^s$ is 
		\begin{align*}
		\beta^s(x,y,t)= (1-t)F(x,y,s\psi (x,y)) +t ( F_1(x,y,s\psi (x,y)), F_2(x,y,s\psi (x,y)),1).
		\end{align*}
		By the stability Lemma \ref{lemma:stability}, it suffices to show that for $(x,y,t)$ lying in a neighborhood of $ D_k\times [0,1]$ the map 
		\begin{align*}
		\beta^0(x,y,t)&=(1-t)F(x,y,0) +t ( F_1(x,y,0), F_2(x,y,0),1)\\
		&= (F_1(x,y,0),F_2(x,y,0), (1-t)F_3(x,y,0)+t) \\
		&= \left(\frac{g(r)}{r}x, \frac{g(r)}{r}y, (1-t)rh(r)\theta +t\right)
		\end{align*}	
		is a locally $L$-Lipschitz quasiconformal embedding with dilatation bound $K$ independent of $k$, where $L>0$ is a constant also independent of $k$. 
		
		The fact that $\beta^0$ is a homeomorphism follows from Lemma \ref{lemma:properties of f}\eqref{item:properties_f(d)}, which implies that there exists an open neighborhood $\Omega_k$ of $ D_k$ such that each vertical line intersects  
		\begin{align*}
		\{F(x,y,0): (x,y)\in \Omega_k \} = \{G(x,y): (x,y)\in \Omega_k \}
		\end{align*}
		at one point. It is also crucial here that the $t$-component $(1-t)rh(r)\theta+t$ of $\beta^0$ is increasing in $t$. Indeed, $1-rh(r)\theta \geq 1/2>0$ by Lemma \ref{lemma:properties of f}\eqref{item:properties_f(d)}, upon choosing a large enough $k\in \N$.
		
		Next, we show that $\beta^0$ is an immersion. The differential matrix at points $(x,0,t)$ is
		
		\begin{align*}
		D\beta^0(x,0,t) = \begin{pmatrix}
		g'(x) & 0 &0 \\
		0 & \frac{g(x)}{x} & 0\\
		0 & (1-t)h(x) & 1 
		\end{pmatrix}.
		\end{align*}
		See also \eqref{Df}. Since $g'(x)g(x)/x \simeq 1$ for all $x\in [0,\infty)$, we see that $D\beta^0(x,y,t)$ is non-singular for all $y$ sufficiently close to $0$ and for all $t$. In particular, by compactness this holds for $(x,y,t)$ lying in a neighborhood of $D_k \times [0,1]$, if $k$ is sufficiently large.
		
		The quasiconformality and Lipschitz bounds follow from the smoothness and Lemma \ref{lemma:g}. Indeed, by the above computation, for sufficiently large $k$, we have $\|D\beta^0(x,y,t)\| \lesssim 1$ for $(x,y,t)$ in a neighborhood of $D_k\times [0,1]$, and also $J_{\beta^0} (x,y,t) \simeq 1$. 
	\end{proof}

	Lemma \ref{lemma:S_1} implies that for each sufficiently large $k\in \N$ and for all sufficiently small $s>0$ the set $\{F(x,y,s\psi(x,y)): (x,y)\in D_k\}$ is the graph of a function. Indeed, the map $F^s|S_1(s)$ is an embedding and is defined linearly on vertical lines. By Lemma \ref{lemma:orientation of F(E)}, this graph lies above the set $\{F(x,y,0): (x,y)\in D_k\}$, which is also a graph of a function contained in $T_k$, by Lemma \ref{lemma:properties of f}\eqref{item:properties_f(d)}. Note that when $(x,y)\in \partial D_k$ the function $\psi(x,y)$ vanishes, so the two graphs are ``glued" along the set $\{F(x,y,0): (x,y)\in \partial D_k\}$. These remarks show that both graphs $\{F(x,y,s\psi(x,y)): (x,y)\in D_k\}$ and $\{F(x,y,0): (x,y)\in D_k\}$ and also the topological ball $F(E^+)$ enclosed by them are contained in $T_k$; see also the remarks preceding Lemma \ref{lemma:orientation of F(E)}.
	
	At this point we fix a value $s \in (0,1]$ sufficiently small as in Lemma \ref{lemma:S_1} and define $F$ on $S_1(s)$ to be $F^s$. Then $F|(E^+\cup S_1(s))$ is a homeomorphism. Set $\varphi = s\psi$, so that $S_1 = S_1(s)$. Moreover, with the same procedure we can define $F$ on $S_1'=\{ z\in T_k: -1\leq  t\leq -\varphi(x,y)\}$. The set $\{F(x,y,-\varphi(x,y)) :(x,y) \in D_k\}$ is again a graph of a function but it lies below $\{F(x,y,0):(x,y)\in D_k\}$. The region enclosed by these two graphs is $F(E^-)$ and it is contained in $T_k$. Pasting together the  maps $F|S_1',F|E^-,F|E^+$, and $F|S_1$ yields a map $F$ on $S_1'\cup E\cup S_1$, which is a homeomorphism  onto its image in $T_k$. Indeed, this map is the result of gluing homeomorphisms that agree on the intersection of their respective domains and whose images do not overlap otherwise.
	
	\begin{rem}\label{remark:boundary}
		A last observation is that $F$ acts as the identity on the arc $\{(1,\theta,0): 0\leq \theta\leq \pi/k\}$ (in cylindrical coordinates), by Lemma \ref{lemma:properties of f}\eqref{item:properties_f(d)}. The linear extensions on vertical lines ensure that $F$ is the identity on $\{(1,\theta,t): 0\leq \theta\leq \pi/k, -1\leq t\leq 1\}$. Moreover, since $\ell_{\pi/k}\times \{0\}$ and $\ell_0 \times\{0\}$ are mapped into $\ell_{\pi/k} \times [-1,1]$ and $\ell_0\times [-1,1]$, respectively by Lemma \ref{lemma:properties of f}\eqref{item:properties_f(d)}, the linear extensions on vertical lines ensure that the sets $\ell_{\pi/k} \times [-1,1]$ and $\ell_0\times [-1,1]$ are mapped onto themselves. These properties are required in order to extend $F$ to all the wedges of the cylinder $T$; see remarks after the statement of Proposition \ref{prop:basic_construction}.
	\end{rem}
	
	\subsection{Definition of \texorpdfstring{$F$}{F} on \texorpdfstring{$S_2$}{S2} and \texorpdfstring{$S_2'$}{S2'}}\label{sec:S_2}
	We recall that $S_2=  D_k \times [1,2]$. Define here for $s\in [0,1]$ a family of maps
	\begin{align*}
	F^s(x,y,t)&= (2-t)F^s(x,y,1)+ (t-1)(x,y,2)\\
	&= (2-t)(F_1(x,y, s\psi(x,y)), F_2(x,y,s\psi(x,y) ), 1) +(t-1)(x,y,2).
	\end{align*}
	Note that $F^s(x,y,1)$ has already been defined, by the definition of $F$ on the region $S_1$. The map $F^s|S_2$ maps the vertical line from $(x,y,1)$ to $(x,y,2)$ linearly onto the vertical line from $(F_1(x,y,s\psi(x,y)), F_2(x,y,s\psi(x,y)), 1)$ to $(x,y,2)$. Also, when $t=2$, this map is the identity.
	
	As in Lemma \ref{lemma:S_1}, we wish to show:
	
	\begin{lemm}\label{lemma:S_2}
		For all sufficiently small $s>0$ the map $F^s|S_2$ is a $K$ quasiconformal embedding, where $K=K(a,b)$ is independent of $k$, and locally $L$-Lipschitz, where $L=L(a)>0$ is independent of $k$. More precisely, for each sufficiently large $k\in \N$ there exists $\delta>0$ such that for all $s\in (0,\delta]$ the map $F^s|S_2$ is a locally $L$-Lipschitz $K$-quasiconformal embedding, with $L=L(a)$ and $K=K(a,b)$ independent of $k$.
	\end{lemm}
	\begin{proof}
		We look again at the function $F^0$ with
		\begin{align*}
		F^0(x,y,t)= (2-t)(F_1(x,y,0), F_2(x,y,0),1) +(t-1)(x,y,2),
		\end{align*}
		which does not depend on $k$. If we show that $F^0$ is a locally Lipschitz quasiconformal embedding (with uniform constants) of a neighborhood of $S_2$, then by the stability Lemma \ref{lemma:stability} $F^s$ will be a quasiconformal embedding of $S_2$ for all small $s$.
		
		We have
		\begin{align*}
		F^0(x,y,t)= \left( \left( \frac{(2-t) g(r)}{r}+ (t-1) \right)x,\left( \frac{(2-t) g(r)}{r}+ (t-1) \right)y, t \right)
		\end{align*}
		in rectangular coordinates. We first show that $F^0$ is injective in a neighborhood of $S_2$. If $F^0(x,y,t)=F^0(x',y',t')$, then $t=t'$, $\theta=\theta'$, and $(2-t)g(r)+(t-1)r= (2-t)g(r')+(t-1)r'$. Since $g$ is increasing, we see that $r=r'$. 
		
		Next, we show that $DF^0$ is non-singular in a neighborhood of $S_2$, which will imply that $F^0$ is an embedding. The Lipschitz bounds and quasiconformality will then follow by smoothness, compactness, and the bounds on $g$ by Lemma \ref{lemma:g}, since we will have $\|DF^0(x,y,t)\|\lesssim 1$ in $S_2$ and $J_{F^0}(x,y,t)\simeq 1$. We have
		\begin{align*}
		DF^0(x,0,t)= \begin{pmatrix}
		(2-t)g'(x) +(t-1) & 0 & x-g(x) \\
		0 & (2-t)\frac{g(x)}{x}+(t-1) & 0 \\
		0 & 0 & 1
		\end{pmatrix}.
		\end{align*}
		This matrix is non-singular for $x\geq 0$ and $1\leq t \leq 2$. If $y$ is sufficiently close to $0$, then $DF^0(x,y,t)$ is also non-singular. Hence, for $(x,y,t)$ in a neighborhood of  $S_2= D_k\times [1,2]$ and a sufficiently large $k$ we have the desired conclusion.
	\end{proof}
	
	We define $F$ on $S_2$ to be $F^s$ for a sufficiently small $s>0$. Note that the choice of $s$ also affects the map $F|S_1'\cup E\cup S_1$ from the previous Section \ref{sec:S_1}. With a similar procedure with define the map $F$ on $S_2'$. The maps $F|S_2$ and $F|S_2'$ can be pasted with the map $F|S_1'\cup E\cup S_1$ from the previous Section \ref{sec:S_1} to yield a homeomorphism from $T_k$ onto $T_k$. In particular, as in Remark \ref{remark:boundary}, the linear extensions on vertical lines guarantee that $\ell_{\pi/k} \times [-2,2]$ and $\ell_{0}\times [-2,2]$ are mapped onto themselves, and the map $F$ restricted to $\partial T_k \setminus (\ell_{\pi/k} \times [-2,2]\cup \ell_0\times [-2,2])$ is the identity. 
	
	For each $\eta>0$ we can choose a large $k\in \N$ such that the map $F$ is $(1+\eta)$-quasiconformal in $E$; see Lemma \ref{lemma:E:qc}. Moreover, $F$ is $K$-quasiconformal and locally $L$-Lipschitz in each of $S_1,S_1',S_2,S_2'$, by Lemma \ref{lemma:S_1} and Lemma \ref{lemma:S_2}; here $K\geq 1$ and $L>0$ are uniform constants, independent of $k$. The boundaries of these sets have finite Hausdorff $2$-measure. Therefore, using Theorem \ref{thm:exceptional_set}, we conclude that $F$ is $K$-quasiconformal and locally $L$-Lipschitz on all of $T_k$. The convexity of $T_k$ implies that $F$ is actually $L$-Lipschitz on $T_k$. The proof of Proposition \ref{prop:basic_construction} is completed. 
	
	\section{Partition and composition scheme} \label{sec:partition_composition}
	Let $A\subset \R^2\times \{0\}$ be a Borel set as in Theorem \ref{thm:main1} and \ref{thm:main2}. In this section we will construct an  effective cover of $A$ by nested cylinders. Then we will describe a scheme for composing rescaled and translated versions of the mappings $F_k$ of Proposition \ref{prop:basic_construction} in order to pass to an appropriate limit and obtain the mapping $f$ described in Theorem \ref{thm:main1} and \ref{thm:main2}.   
	
	For this section, we recall the Vitali covering theorem. See for instance the book of Heinonen \cite[Theorem 1.6]{Hei:01}.
	
	\begin{thm}[Vitali covering theorem] \label{thm:vitali}
		Let $A \subset \mathbb{R}^n$, and let $\mathcal{F}$ be a collection of closed balls centered in $A$ such that $\inf\{r>0: B(z,r) \in \mathcal{F}\} = 0$ for all $z \in A$. There exists a countable subcollection $\mathcal{G} \subset \mathcal{F}$ of mutually disjoint balls such that $ A \setminus \bigcup_{B \in \mathcal{G}} B$ has measure zero.
	\end{thm}
	
	We refer to a cover of $A$ satisfying the conclusions of Theorem \ref{thm:vitali} as a {\it Vitali cover} of $A$. To simplify the notation, in what follows $\R^2$ is identified with its embedding $\R^2\times\{0\}$ in $\R^3$.

	Let $\mathcal{F}$ be the set of closed balls in $\mathbb{R}^2$ centered in $A$ of diameter at most $\lambda$, where $\lambda>0$ is a fixed parameter. Consider a Vitali cover $\mathcal{G} = \{G_1, G_2, \ldots \} \subset \mathcal{F}$ of $A$, with $G_j = \overline{B}(z_j,r_j)$. Let $U_j = G_j \times [-2r_j, 2r_j]$ and $\varphi_j$ be a similarity map from the cylinder $T$ to $U_j$. Note that the Vitali cover $\mathcal{G}$ may contain finitely many balls, but to simplify the notation we assume that $\mathcal{G}$ is countably infinite, as are the other Vitali covers below. 	
	
	For all $m \in \mathbb{N}$, choose a value $k_m \in \mathbb{N}$ sufficiently large so that $\eta(k_m)$ in Proposition \ref{prop:basic_construction} is less than $2^{-m}$. Let $g_1\colon  \mathbb{R}^3 \rightarrow \mathbb{R}^3$ be the mapping defined by 
	$$g_1(z) = \left\{ \begin{array}{ll} \varphi_j \circ F_{k_1} \circ \varphi_j^{-1}(z) & \text{ if } z \in U_j \\ z & \text{ if } z \notin \bigcup_j U_j\end{array} \right. ,$$
	where $F_{k_1}$ is  the mapping from Proposition \ref{prop:basic_construction} corresponding to the parameter $k_1$. The mapping $g_1$ is an $L$-Lipschitz, $K$-quasiconformal homeomorphism for the values $L,K$ in Proposition \ref{prop:basic_construction}. To verify this, first note that these properties hold locally for all $z$ lying in the interior of some cylinder $U_j$, $j \in \mathbb{N}$. Moreover, if $z \notin  (\bigcup_j U_j) \cup \R^2$, then $g_1$ is the identity map in a neighborhood of $z$. That $g_1$ is globally quasiconformal and Lipschitz now holds by Theorem \ref{thm:exceptional_set}, since $\bigcup_j \partial U_j$ and $\R^2$ have $\sigma$-finite Hausdorff 2-measure.

	We repeat this procedure inductively as follows. Let $J$ denote the multi-index $(j_1, \ldots, j_m)$ of length $m$. Assume we have a ball $G_J = \overline{B}(z_J, r_J) \subset \mathbb{R}^2$ with corresponding cylinder $U_J = G_J \times [-2r_J, 2r_J]$ and similarity mapping $\varphi_J\colon T \rightarrow U_J$. Let $\mathcal{F}_J$ denote the family of closed balls $\overline{B}(z,r)$ such that $z \in G_J$ and $\overline{B}(z,r) \times [-2r,2r] $ is contained in the open neighborhood $V_J = \varphi_J(V_{k_m})$ of \eqref{item:basic_construction(c)} in Proposition \ref{prop:basic_construction}. In particular, observe that $\overline{B}(z,r) \cap \varphi_J(Z_{k_m}) = \emptyset$, where $Z_{k_m}$ is the set described prior to Proposition \ref{prop:basic_construction}. We also require that $r\leq 2^{-m-1}$. Choose a Vitali cover $\mathcal{G}_J = \{G_{J,1}, G_{J,2}, \ldots \} \subset \mathcal{F}_J$ of $G_J \setminus \varphi_J(Z_{k_m})$, where $G_{J, j_{m+1}} = \overline{B}(z_{J, j_{m+1}}, r_{J, j_{m+1}})$ and $U_{J, j_{m+1}} = G_{J, j_{m+1}} \times [-2r_{J, j_{m+1}}, 2r_{J, j_{m+1}}]$.  Since the set $\varphi_J(Z_{k_m})$ has Lebesgue 2-measure zero, $\mathcal{G}_J$ also covers $G_J$ up to a set of Lebesgue 2-measure zero. Finally, let $\varphi_{J, j_{m+1}}: T \rightarrow U_{J, j_{m+1}}$ be a similarity mapping. 
	
	As above, let $g_m\colon \mathbb{R}^3 \rightarrow \mathbb{R}^3$ be the mapping defined by $g_m(z) = \varphi_J \circ F_{k_m} \circ \varphi_J^{-1}(z)$ if $z \in U_J$ for some multi-index $J$ of length $m$ and $g_m(z) = z$ otherwise. Like $g_1$, each of the mappings $g_m$ is $K$-quasiconformal and $L$-Lipschitz. Let $h_m = g_1 \circ g_2 \circ \cdots \circ g_m$.  Observe that the smaller-scale mappings are applied first in this composition.
	
	\begin{lemm}
		For all $m \in \mathbb{N}$, $h_m$ is $(eK)$-quasiconformal, where $K$ is the constant in Proposition \ref{prop:basic_construction}. 
	\end{lemm}
	\begin{proof}
		Let $z \in \mathbb{R}^3 \setminus (\mathbb{R}^2\times \{0\})$ and fix $m\in \mathbb{N}$. Let $p \in \{0, \ldots, m\}$ be the largest value with the property that there is a  (possible empty) multi-index $J_p = (j_1, \ldots, j_p)$ such that $z \in U_{J_p}$. For each $q \in \{1, \ldots, p\}$, let $J_q = (j_1, \ldots, j_q)$. In this manner we obtain a (possibly empty) finite sequence $\mathcal{J}= \{J_1, \ldots, J_p\}$, where each $J_q$ is a multi-index of length $q$, such that $z \in U_J$ if $J \in \mathcal{J}$ and $z \notin U_J$ if $J \notin \mathcal{J}$. If $\mathcal{J}$ is empty, then $h_m$ {is the identity map in a neighborhood of $z$}, so we ignore this case. By the construction of the Vitali covers, $U_{J_q} \subset V_{J}$ for all $J \in \{J_1, \ldots, J_{q-1}\}$.  This implies $g_q|U_{J_q}$ is $(1+\eta(k_q))$-quasiconformal for all $q \in \{1, \ldots,p-1\}$, while $g_p$ is globally $K$-quasiconformal. Thus $h_{m}|U_{J_p}$ is $K'$-quasiconformal for
		\begin{equation} \label{equ:dilatation_bound}
		K' \leq (1+\eta(k_1)) \cdots (1+\eta(k_{p-1}))K.  
		\end{equation}
		
		Let $p'=p-1$. A simple way to bound the quantity on the right-hand side of \eqref{equ:dilatation_bound} is to use the inequality of arithmetic and geometric means to obtain
		$$(1+\eta(k_1)) \cdots (1+\eta(k_{p'}))\leq \left(\frac{p' + \eta(k_1) + \cdots + \eta(k_{p'})}{p'}\right)^{p'} \leq \left(\frac{p' + 1}{p'}\right)^{p'}, $$
		where the right-most term increases to the constant $e$ as $p'$ goes to infinity. Thus $h_m|U_{k_p}$ is $(eK)$-quasiconformal. 
		
		We have shown that $h_m$ is $(eK)$-quasiconformal on $\mathbb{R}^3 \setminus (\mathbb{R}^2\times \{0\})$. The result now follows from Theorem \ref{thm:exceptional_set}.
	\end{proof}
	
	It is known (see for instance Corollary 19.5 in \cite{Vais:71}) that any family of quasiconformal mappings with uniformly bounded distortion is equicontinuous and hence normal, provided each mapping in this family agree on a three-point set. Since each mapping $h_m$ fixes the boundary of the balls in $\mathcal{G}$, this condition is clearly satisfied. Thus there is a subsequence of $\{h_m\}_m$ (actually, easily seen to be $\{h_m\}_m$ itself) that converges uniformly on compact sets to a $(e^CK)$-quasiconformal mapping $f\colon \mathbb{R}^3 \rightarrow \mathbb{R}^3$. It is this mapping $f$ that satisfies the conclusions of Theorem \ref{thm:main1}, as we show in the next section.

	\section{Inverse absolute continuity} \label{sec:inverse_absolute_continuity}
	
	In this section, we verify that $f$ as defined at the end of Section \ref{sec:partition_composition} satisfies the conclusions of Theorem \ref{thm:main1}. The proof of this fact is similar to that of the corresponding result in the paper \cite{Rom:18}. Then, in Section \ref{sec:proof_thm2}, we modify the argument to yield Theorem \ref{thm:main2}.
	
	Let us assume first that $A$ is bounded, or more precisely that $\bigcup \mathcal{G}$ is contained in the ball $B(0,R)$ for some $R>0$.  
	
	Let $\mathcal{G}^1 = \mathcal{G}$ and $\mathcal{G}^m = \bigcup_{J \in \mathbb{N}^{m-1}} \mathcal{G}_J$ for all integers $m \geq 2$; that is, $\mathcal{G}^m$ is the collection of all balls $G_J$ of level $m$. Let $U= \bigcap_m \left( \bigcup \mathcal{G}^m\right) \subset \mathbb{R}^2$. In words, $U$ is the set of points in $\mathbb{R}^2$ that lie in a ball of $\mathcal{G}^m$ for infinitely many $m$.  Observe that $U$ is a full-measure subset of $\bigcup \mathcal G$, since the balls of each collection $\mathcal G^m$ cover $\bigcup \mathcal G$ up to a set of measure zero. Define for each $m \in \mathbb{N}$ a random variable $X_m$ on $U$ (equipped with Lebesgue 2-measure, normalized so that the measure of $U$ is 1) as follows. Let $G^m(z)$ be the ball $G_J$  of level $m$ from Section \ref{sec:partition_composition} that contains $z$, with $\varphi_z^m \colon \mathbb{D} \rightarrow G^m(z)$ the restriction of the corresponding similarity map $\varphi_J$. Such a ball $G_J$ exists for all $z \in U$ by the definition of $U$ and is unique. Then define $X_m$ by 
	$$X_m(z) = \left\{ \begin{array}{ll} b & \text{ if } z \in \varphi_z^m(B(0,a)) \\ L & \text{ if } z \in \varphi_z^m(\mathbb{D} \setminus B(0,a)) \end{array} \right. .$$
	Here $L$ is the Lipschitz constant of the map $g_m$, as in Section \ref{sec:partition_composition}. We recall that $L=L(a)$ depends only on $a$; see Proposition \ref{prop:basic_construction}.

	An important observation here is that each $X_m$ takes the value $b$ with probability $a^2$ and the value $L$ with probability $1-a^2$.  This is because, for a given multi-index $J$, $\mathcal{G}_J$ is a collection of disjoint sets that covers $G_J$ up to a set of Lebesgue 2-measure zero. In particular, the geometric mean of $X_m$ is 
	\begin{equation} \label{equ:mu}
	\mu = b^{a^2}L^{1-a^2} .
	\end{equation} 
	Recall that the geometric mean of $X_m$ is $e^{\mu'}$, where $\mu'$ is the usual arithmetic mean of $\log X_m$.  	Moreover, the random variables $X_m$ are all mutually independent. It is crucial here that each ball of $\mathcal G_J$ lies either in $\varphi_J(B(0,a)\times \{0\})$ or in the annulus $ \varphi_J((\D\setminus B(0,a))\times \{0\})$; see the construction in Section \ref{sec:partition_composition}. 
	
	We can now specify the requirement on the parameters $a,b$, which appear in Proposition \ref{prop:basic_construction}. The idea is that we want the contraction on the interior ball to have a larger influence than the expansion on the outer annulus. This is the case if $\mu<1$. To this end, first choose an arbitrary $a \in (0,1)$, and then choose $b \in (0,1)$ sufficiently small so that $\mu<1$ in \eqref{equ:mu}. Moreover, in Section \ref{sec:proof_thm2}, we will require  $a$ to satisfy $a^2 > 1/2$ and $b$ to be sufficiently small so that $b \leq 1/L$.

	Define next the random variable $Y_k$ on $U$ by $Y_k = X_1 \cdots X_k$. The strong law of large numbers implies that the sequence $(Y_k^{1/k})$ converges almost surely to $\mu$. The key observation here is that the random variable $X_m$ is a bound on the pointwise Lipschitz constant for the mapping $g_m$, while $Y_k$ is a bound on the pointwise Lipschitz constant for $h_k$. For the latter claim, one has to note that if a ball  $G_{J,j_{m+1}}$ of $\mathcal G_{J}$ is contained in $\varphi_J(B(0,a)\times \{0\})$, then the corresponding cylinder  $U_{J,j_{m+1}}$ is contained in the open set $V_J$, on which the map $g_m$ (where $J=(j_1,\dots,j_m)$) is a scaling by the factor $b$; see Proposition \ref{prop:basic_construction} and the construction in Section 3.
	
	We want to estimate the size of the set $f(A')$, where $A'$ is a full-measure subset of $U$. Pick $c \in (1, 1/\mu)$. Let $\mathcal{E}_m$ be the collection of balls $B= B(z,r) \in \mathcal{G}^m$ such that $Y_{m-1}|B \leq (c\mu)^{m-1}$ (observing that $Y_{m-1}$ is constant on $B$). Then it follows that $\diam f(B) \leq 2r(c\mu)^{m-1}$ for all $B \in \mathcal{E}^m$.  
	
	Let $\delta(m) = 2^{-m}$ and $E_m = \bigcup \mathcal{E}_m$. Observe that, since $\diam B \leq 2r \leq 2^{-m}$, we must have $\diam f(B) \leq 2^{-m}(c\mu)^{m-1} \leq \delta(m)$ for all $B \in \mathcal{E}_m$. We can estimate the 2-dimensional Hausdorff $\delta(m)$-content of $f(E_m)$ as follows:
	$$\mathcal{H}_{\delta(m)}^2(f(E_m)) \leq \sum_{B \in\mathcal{E}_m} r(B)^2 (c\mu)^{2(m-1)} \leq R^2 (c\mu)^{2(m-1)}.$$
	In the right-most inequality, we have used the fact that $\sum_{B\in \mathcal E_m} r(B)^2 \leq R^2$, from the assumption that $A$ is bounded. From this, we have that
	$$\mathcal{H}_{\delta(m)}^2\left(\bigcup_{k=m}^\infty f(E_k)\right) \leq R^2 \sum_{k=m}^\infty (c\mu)^{2(k-1)} = \frac{ R^2(c\mu)^{2(m-1)}}{1-(c\mu)^2}.$$
	Now let 
	\begin{equation}\label{equ:A'} 
	A'= \bigcap_{m=1}^\infty \bigcup_{k=m}^\infty E_k.
	\end{equation} 
	Then for all $m \in \mathbb{N}$, 
	$$\mathcal{H}_{\delta(m)}^2(f(A')) \leq \frac{ R^2(c\mu)^{2(m-1)}}{1-(c\mu)^2}.$$
	This shows that $\mathcal{H}^2(f(A')) = 0$  by letting $m \rightarrow \infty$. By the almost sure convergence of $(Y_k^{1/k})$ to $\mu$, the set $\bigcup_{k=m}^\infty E_k$ is a full-measure subset of $U$, and so $A'$ is also a full-measure subset of $U$. Since $\mathcal{H}^2(A \setminus U) = 0$, it follows that $A'$ has full measure in $A$. This completes the proof in the case that $A$ is bounded.

	The general case now follows easily by writing $A$ as the countable union of bounded sets. To be more precise about this, for each $k \in \mathbb{N}$, let $\mathcal{G}(k)$ denote the subcollection of $\mathcal{G}$ comprising those balls that are contained in $B(0,k)$ but not in $B(0,k-1)$. Then the preceding argument yields  a full-measure subset $A'(k) \subset A \cap (\bigcup \mathcal{G}(k))$ with the property that $f(A'(k))$ has Hausdorff 2-measure zero. Let $A' = \bigcup_{k=1}^\infty A'(k)$. Then $A'$ has the property that $f(A')$ has Hausdorff 2-measure zero, and moreover $A'$ has full measure in $A$. This completes the proof of Theorem \ref{thm:main1}.
	
	\subsection{Proof of Theorem \ref{thm:main2}} \label{sec:proof_thm2}
	
	Next, we show how to modify the preceding proof to obtain Theorem \ref{thm:main2} instead. This is the same idea as in Section 6 of \cite{Rom:18}.  As before, begin with the case that $A$ is bounded, with $\bigcup \mathcal{G}$ contained in the ball $B(0,R)$ for some $R>0$. Using the above notation, let $\widetilde{X}_m$ be the random variable on $U$ defined by
	$$\widetilde{X}_m(z) = \left\{ \begin{array}{ll} 1 & \text{ if } z \in \varphi_z^m(B(0,a)) \\ -1 & \text{ if } z \in \varphi_z^m(\mathbb{D} \setminus B(0,a)) \end{array} \right. .$$  
	Again, take Lebesgue 2-measure as the probability measure on $U$, normalized so that the measure of $U$ is 1. Let $\widetilde{Y}_m = \widetilde{X}_1 + \cdots + \widetilde{X}_m$. Now modify the composition scheme in Section \ref{sec:partition_composition} as follows. Fix $M \in \mathbb{N}$. Say $z \in G_J$ for some multi-index $J = (j_1, \ldots, j_m)$ of length $m$. If $\widetilde{Y}_{m'}(z) = -M$ for some $m' \leq m$, then modify the definition of $g_m$ from Section \ref{sec:partition_composition} so that $g_m(z) = z$ for all $z \in U_J$. Let $A''$ be the subset of $A'$ on which this modification does not occur, where $A'$ is as in \eqref{equ:A'}.  Recall that $A'$ has full measure in $A$.
	
	This modification guarantees that the resulting mapping is Lipschitz with Lipschitz constant $L^M$, since we have chosen $b$ so that $b \leq 1/L$. This is because $Y_m(z) = b^j L^{m-j}$ for some $j$ satisfying $j \geq m-j-M$ whenever $g_m(z)$ is not the identity on $z$. But then 
	$$Y_m(z) \leq \left(\frac{1}{L}\right)^j {L}^{m-j} = L^{m-2j} \leq L^{M}.$$  
	
	Since the mapping is unchanged on $A''$, it follows that $\mathcal{H}^2(f(A'')) = 0$. So the final step is to give a bound on $\mathcal{H}^2(A'\setminus A'')$, as claimed in Theorem \ref{thm:main2}. 
	
	Consider $\widetilde{Y}_m$ as a random walk that steps up with probability $p = a^2$ and steps down with probability $q = 1-a^2$, and terminates if $\widetilde{Y}_m=-M$.  Assume first that $M=1$. If $p \geq 1/2$, this random walk terminates with probability $P_1 = (1-p)/p$. To see this, let $r$ be the probability of this event occurring. Then $r$ satisfies the recursive relationship $r = q+pr^2$. This has solutions $r=q/p$ and $r=1$. However, since $\widetilde{Y}_m/m$ converges almost surely to $p-q$, this implies that $r \neq 1$. We conclude that $r = q/p = (1-p)/p$. See the book of Klenke \cite[Sec. 17.5]{Kle:08} for more detail.  For arbitrary $M \in \mathbb{N}$, it follows by induction that the corresponding random walk terminates with probability $P_M = (1-p)^M/p^M$. 
	
	Interpreting this probabilistic statement back in terms of the measure, we have shown that $\mathcal{H}^2(A'\setminus A'') \leq P_M\mathcal{H}^2(U) = P_M\mathcal{H}^2(A')$. Recall our previous assumption that $p = a^2 > 1/2$, justifying this conclusion. Taking $M$ sufficiently large, we can ensure that $P_M$ is arbitrarily close to zero. For the given $\kappa>0$ in the statement of Theorem \ref{thm:main2}, we thus obtain a set $A''$ with  $\mathcal{H}^2(A'\setminus A'') < \kappa$. Since $A'$ has full measure in $A$, it follows as well that $\mathcal{H}^2(A\setminus A'') \leq \mathcal{H}^2(A'\setminus A'') < \kappa$. This concludes the bounded case.

	The case where $A$ is not bounded proceeds again by writing $A$ as a countable union of bounded sets. Following the notation in the corresponding portion of the proof of Theorem \ref{thm:main1}, write $\mathcal{G}$ as the disjoint union $\mathcal{G} = \bigcup_k \mathcal{G}(k)$, with $A'(k) \subset A \cap (\bigcup \mathcal{G}(k))$ as described there. For each $k \in \mathbb{N}$, apply the truncation procedure of this section to modify $f$ on a subset of $\bigcup \mathcal{G}(k)$. In doing so, use a sufficiently large value $M(k)$ so that the resulting set $A''(k) \subset A'(k)$ satisfies $\mathcal{H}^2(A'(k)\setminus A''(k)) < 2^{-k}\kappa$. The set $A'' = \bigcup_{k=1}^\infty A''(k)$ then satisfies $\mathcal{H}^2(f(A''(k))) = 0$ and $\mathcal{H}^2(A'\setminus A'') = \mathcal{H}^2(A\setminus A'') < \kappa$. Moreover, $f$ is $L^{M(k)}$-Lipschitz on the set $\bigcup \mathcal{G}(k)$. This shows that $f$ is locally Lipschitz. 
	
	\bibliographystyle{alpha}
	\bibliography{EuclideanIACbiblio} 
	
\end{document}